\documentclass[11pt,twoside]{amsart}
\usepackage{amssymb}
\usepackage{color}
\usepackage[small,nohug,heads=vee]{diagrams}
\diagramstyle[labelstyle=\scriptstyle]
 \makeatletter
    
\@addtoreset{equation}{section}
\@addtoreset{equation}{subsection}
  \makeatother
\newtheorem{thm}{Theorem}[section]
\newtheorem{subthm}[thm]{Theorem}

\newtheorem{subprop}[thm]{Proposition}
\newtheorem{lemma}[thm]{Lemma}
\newtheorem{sublemma}[thm]{Lemma}
\newtheorem{cor}[thm]{Corollary}
\newtheorem{subcor}[thm]{Corollary}

\theoremstyle{definition}
\newtheorem{defn}[thm]{Definition}
\newtheorem{subdefn}[thm]{Definition}

\newtheorem{rem}[thm]{Remark}
\newtheorem{subrem}[thm]{Remark}

\newcommand\op{\operatorname}
\newcommand\Coh{\op{Coh}}
\newcommand\Coinv{\op{Coinv}\,}
\newcommand\diag{\op{diag}}
\newcommand\End{\op{End}\,}
\newcommand\ext{\op{ext}}
\newcommand\Ext{\op{Ext}}
\newcommand\Gen{\op{Gen}}
\newcommand\GL{\op{GL}}
\newcommand\GHilb{G\op{-Hilb}}
\newcommand\Ginv{G\op{-inv}}
\newcommand\Hom{\op{Hom}}
\newcommand\id{\op{id}}
\newcommand\image{\op{Im}}
\newcommand\Irr{\op{Irr}}
\newcommand\Ker{\op{Ker}\,}
\newcommand\LF{\op{LF}}
\newcommand\Lotimes{\overset{\bL}{\otimes}}
\newcommand\nat{\op{nat}}
\newcommand\otimesk{\otimes_k }
\newcommand\Qis{\op{Qis}}
\newcommand\red{\op{red}}
\newcommand\Rfl{\op{Rfl}}
\newcommand\Rhom{\bR Hom}
\newcommand\rank{\op{rank}}
\newcommand\Sing{\op{Sing}\,}
\newcommand\SL{\op{SL}}
\newcommand\Soc{\op{Soc}}
\newcommand\SSoc{\op{SSoc}}
\newcommand\Spec{\op{Spec}\,}
\newcommand\Supp{\op{Supp}\,}
\newcommand{\chom}{{\cH}om}
\newcommand{\cext}{{\cE}xt}

\newcommand{\wD}{\widetilde{D}}
\newcommand{\wV}{{\widetilde V}}
\newcommand{\bA}{{\mathbf A}}
\newcommand{\bC}{{\mathbf C}}
\newcommand{\bD}{{\mathbf D}}
\newcommand{\bG}{{\mathbf G}}
\newcommand{\bL}{{\mathbf L}}
\newcommand{\bM}{{\mathbf M}}
\newcommand{\bP}{{\mathbf P}}
\newcommand{\bR}{{\mathbf R}}
\newcommand{\bZ}{{\mathbf Z}}
\newcommand{\cA}{{\mathcal A}}
\newcommand{\cC}{{\mathcal C}}
\newcommand{\cE}{{\mathcal E}}
\newcommand{\cF}{{\mathcal F}}
\newcommand{\cG}{{\mathcal G}}
\newcommand{\cH}{{\mathcal H}}
\newcommand{\cI}{{\mathcal I}}
\newcommand{\cK}{{\mathcal K}}
\newcommand{\cM}{{\mathcal M}}
\newcommand{\cN}{{\mathcal N}}
\newcommand{\cO}{{\mathcal O}}
\newcommand{\cV}{{\mathcal V}}
\newcommand{\cW}{{\mathcal W}}
\newcommand{\cZ}{{\mathcal Z}}
\newcommand{\fm}{{\mathfrak m}}
\newcommand{\fn}{{\mathfrak n}}
\begin{document}
\title[McKay correspondence]
{Extended  McKay correspondence 
for quotient surface singularities}
\author{Akira Ishii}
\address{Division of Mathematical and Information Sciences,
Hiroshima University, 1-7-1 Kagamiyama,
Higashi-Hiroshima, 739-8521, Japan} 
\email{akira141@hiroshima-u.ac.jp}
\author{Iku Nakamura, \today}
\address{Department of Mathematics, 
Hokkaido University, Sapporo, 060-0810, Japan} 

\email{nakamura@math.sci.hokudai.ac.jp}
\thanks{Research was supported in part by the Grants-in-aid 
(15K04819, S-23224001) for Scientific Research, 
JSPS}
\date{\today}
%
\begin{abstract}
Let $G$ be a finite subgroup of $\GL(2)$ 
acting on $\bA^2\setminus\{0\}$ freely. 
The $G$-orbit Hilbert scheme
$\GHilb(\bA^2)$ is a minimal resolution 
of the quotient $\bA^2/G$ by \cite{Ishii02}. 
We determine the generator sheaf 
of the ideal defining the universal 
$G$-cluster over $\GHilb(\bA^2)$, 
which somewhat 
strengthens 
the well-known McKay correspondence 
for a finite subgroup of $\SL(2)$. 
We also study the quiver structure of $\GHilb(\bA^2)$ 
at every $G$-cluster $O_{Z_y}=O_{\bA^2}/I_y$ in terms of      
a collection of sort of minimal $G$-submodules of $O_{Z_y}$ 
(called mono-special $O_{\bA^2}$-submodules) and 
generating $G$-submodules of $I_y$.  
\end{abstract}
\maketitle

\section{Introduction.}
\label{sec:Introduction} 
For an algebraically closed field $k$, let 
$G$ be any finite small subgroup of $\GL(2,k)$, that is, 
a finite subgroup with no pseudo-reflections, 
or equivalently, a finite subgroup of $\GL(2,k)$ 
acting on $\bA^2_k$ 
 with the origin the unique fixed point of it. 
Throughout this article 
we assume that the characteristic of $k$ 
is prime to the order $|G|$ of $G$.\par
The $G$-orbit Hilbert scheme $\GHilb(\bA_k^2)$ 
is the scheme parameterizing all the $G$-invariant 
zero-dimensional subschemes of $\bA^2_k$ of length $|G|$, each with 
structure sheaf isomorphic 
to the group algebra $k[G]$ of $G$ as a $k[G]$-module.
 It is by \cite{Ishii02} 
the minimal resolution of the quotient $\bA_k^2/G$.

For any geometric point $y$ in the exceptional set of the resolution, 
let  $\Gen(I_y)$ be 
the minimal $k[G]$-(sub)module generating the ideal $I_y \subset O_{\bA^2_k}$ corresponding to $y$. 
One of the purposes of this article is to prove the following:
\begin{thm}
\label{thm:thm}
Let $G$ be a finite small subgroup of $\GL(2,k)$, 
and  $E$ the exceptional set of the minimal resolution 
$\GHilb(\bA_k^2)$ of $\bA_k^2/G$. Then 
the generator sheaf 
of the ideal defining the universal 
$G$-cluster, 
that is, the union of all 
$\Gen(I_y)$ over $E$ is an $O_F$-module $G$-isomorphic to 
$$
(\bigoplus_{\rho\neq \rho_0}O_{E(\rho)}(-1)
\otimesk \rho)\bigoplus O_F(-F)\otimesk \rho_0
$$where 
 $\rho$ ranges over the set of all non-trivial irreducible representations 
of $G$, special in the sense of Definition~\ref{subdefn:special repres}, $E(\rho)$ is an irreducible component of $E$ associated to $\rho$ and $F$ is 
the fundamental divisor of $E=F_{\red}$.
\end{thm}

See Theorem~\ref{thm:mckay isom} and  Corollary~\ref{cor:global local isom}.
Our main ingredient for the proof 
is the derived category method adopted by 
\cite{BKR01} and \cite{Ishii02}. See also \cite{KapranovVasserot}.

The article also aims at studying 
the special McKay quiver associated with $\GHilb(\bA_k^2)$  
in terms of its    
``mono-special $O_U$-submodules'' of $O_{Z_y}$ (see Definition \ref{def:mono-special})
and generators of $I_y$
at every $G$-cluster $O_{Z_y}=O_U/I_y$.

Here is an outline of the article. 
In Section~\ref{sec:reflexive and full}, 
we recall reflexive modules and full modules. 
In Section~\ref{sec:global McKay}, we review 
local McKay correspondence 
for two-dimensional quotient singularities 
from \cite{Ishii02}. 
{Then} 
we formulate and prove {its global version, global in the sense that it describes a family of the local versions over the exceptional set}
(Theorem~\ref{thm:thm}). 
In Section~\ref{sec:extensions and cup}, we discuss 
the connection between tensor product with the natural representation,  
extensions of the sheaf $O_{Z_y}$, and cup products 
of $\Ext$ groups. As an application, we prove  
Theorem~\ref{thm:mckay tensor prod}. \par
 \par 
In Section~\ref{sec:semi-orth proj}, we study 
the semi-orthogonal projection of $O_0\otimesk \rho$ 
{to the essential image of $D^b(\Coh \GHilb(\bA^2))$}
for every special representation of $G$. 
For every quotient surface 
singularity arising from a subgroup of $\GL(2)$, 
we modify the notion of socles 
of $O_{Z_y}$ so as 
to make more coherent to analysis of exceptional sets.
This is done in Theorem~\ref{thm:socle-special}, where 
mono-special $O_U$-submodules of $O_{Z_y}$ are introduced. \par
In Section~\ref{sec:reconst alg and deformations}, we recall from \cite{Craw12} and \cite{Wemyss11} the special McKay quiver,  which describes the reconstruction algebra of the singularity $X$.
We then generalize 
Theorem~\ref{thm:mckay tensor prod} to every (small) finite subgroup of 
$\GL(2,k)$.  This describes a precise connection 
between a mono-special $O_U$-submodule $V(\sigma)$ of $O_{Z_y}$ 
and a generator $G$-submodule 
$W(\rho)\subset \Gen(I_y)$ for every individual $G$-cluster $Z_y$ 
and every pair $(\sigma,\rho)$ 
of special representations. 
Finally we explain all of the above for two examples, 
a quotient singularity arising from a 
cyclic group $\langle \frac{1}{12}(1,5) \rangle$ 
of order 12 and  
the binary dihedral case $D_5$.   \par
 We are very grateful 
to  John McKay for his encouragement 
and interest in the subject. We are very grateful 
also to Alastair Craw for his encouragement and 
suggestion of better terminology.

\section{Reflexive modules and torsion free pullbacks}
\label{sec:reflexive and full}
\subsection{Derived category}
\begin{subdefn}
Let $X$ be a scheme over a fixed field $k$.
Let $Coh(X)$ be the category of coherent $O_X$-modules and 
$K(Coh(X))$ the category consisting of the bounded complexes of 
objects and morphisms in $Coh(X)$. 

We define a morphism $f:P^{\bullet}\to Q^{\bullet}$ in $K(Coh(X))$ to be 
{\em a quasi-isomorphism} if $f$ induces an isomorphism on cohomology. \par
The derived category ${D(X)=}D(Coh(X))$ is the localization of $K(Coh(X))$ 
at $\Qis$ (the monoid of quasi-isomorphisms.), that is, it is the category 
$K(Coh(X))$ modulo equivalence defined by $\Qis$. See 
\cite[Def., pp.~49--50]{Hartshorne66}. 
Similarly we define $Coh_c(X)$ 
to be the category of coherent $O_X$-modules
with complete supports, $K(Coh_c(X))$ the category consisting of bounded complexes of objects and morphisms in $Coh_c(X)$, and  
$D_c(X)$ the derived category of $K(Coh_c(X))$. 
\end{subdefn}

\subsection{Reflexive $O_X$-modules and full $O_Y$-modules}
This subsection is taken from \cite{AV85}, \cite{EK85}, and 
\cite{Wunram88}. 
Let $Z$ be a scheme of finite type over $k$, 
$\cF$ a coherent $O_Z$-module on $Z$. 
Then $\cF$ is defined to be a {\it reflexive} $O_Z$-module 
iff $\cF^{\vee\vee}\simeq \cF$, 
where $\cF^{\vee}=\chom_{O_Z}(\cF,O_Z)$. 

\begin{sublemma}
\label{sublemma:reflexive}Let $Z$ be a normal surface, 
$Z'=Z\setminus\Sing(Z)$ and $i:Z'\hookrightarrow Z$ 
the inclusion. 
The following is true.
\begin{enumerate}
\item  
any torsion free module over a discrete valuation ring is free, hence, 
any torsion free sheaf on a  nonsingular curve is locally free. 
\item 
if $Z$ is nonsingular, any reflexive $O_Z$-module 
is locally free,
\item if $\cF$ is 
a reflexive $O_Z$-module, then  
$i_*(i^*\cF)=\cF$, and it is uniquely determined 
by its restriction to $Z'$, 
\item if $\cF$ is a locally free $O_{Z'}$-module, 
then $i_*(\cF)$ is a 
reflexive $O_Z$-module,
\item $\cG^{\vee}$ is reflexive
for any finite $O_Z$-module $\cG$.
\end{enumerate}
\end{sublemma}
\begin{proof}
See \cite[Corollary 1.4]{Hartshorne80} for the parts (1) and (2). 
See \cite[Proposition~1.6]{Hartshorne80} for the part (3). 
See \cite[Corollary 1.2]{Hartshorne80} for the part (5). 
The part (4) is proved as follows. Let $\cG=i_*(\cF)$.  
Since $\cF=i^*(\cG)$ is locally free on $Z$, we have 
$\cG\simeq i_*(i^*\cG)\simeq i_*(i^*(\cG^{\vee\vee}))\supset \cG^{\vee\vee}$, 
from which (4) follows.
\end{proof}

{The following summarizes \cite[Lemma 2,1, 2.2]{E85}}.
\begin{sublemma}
\label{sublemma:basics}
Let $X$ be an affine normal surface with rational singularities 
and $f:Y\to X$ the minimal resolution of $X$. 
\begin{enumerate}
\item Let $M$ be a reflexive $O_X$-module, $\cM$ 
the torsion free pullback of $M$ to $Y$, that is, 
$\cM=f^*M/O_Y\text{-torsions}$. Then {$f_*(\cM)=M$, and 
\begin{enumerate}
\item[(i)] 
 $\cM$ is locally free, 
\item[(ii)] $\cM$ is 
 generated by global sections, 
\item[(iii)] $R^1f_*(\cM^{\vee}\otimes_{O_Y} \omega_Y)=0$.
\end{enumerate}}
\item Conversely if a coherent $O_Y$-module $\cM$ satisfies 
{(i) and (iii),} 
then $M:=f_*(\cM)$ is a reflexive $O_X$-module. 
\end{enumerate}
\end{sublemma}

\begin{subdefn}
We call an $O_Y$-module $\cM$ 
a {\it full} $O_Y$-module (or simply a full sheaf following \cite{E85}) 
if the conditions {(i)--(iii)} are satisfied. \end{subdefn}

The following is a corollary of Lemma~\ref{sublemma:basics} (and 
Lemma~\ref{sublemma:reflexive}):
\begin{subcor}
\label{subcor:reflexive and full}
Under the same notation as in Lemma~\ref{sublemma:basics}, 
there is a bijective correspondence between the following sets
\begin{enumerate}
\item[(i)] the set of (indecomposable) reflexive $O_X$-modules $M$, 
\item[(ii)] the set of (indecomposable) full $O_Y$-modules $\cM$. 
\end{enumerate}
\end{subcor}

\begin{subcor}\label{subcor:uniqueness of a full sheaf}
Under the same notation as in Lemma~\ref{sublemma:basics},
every full $O_Y$-module $\cM$ is determined 
by its restriction to $Y\setminus f^{-1}(\Sing(X))$.
\end{subcor}
\begin{proof}
Clear from Lemma~\ref{sublemma:reflexive}~(3).
\end{proof}

\begin{subcor}
\label{cor:KY is full}
The invertible sheaf $K_Y$ is 
a full sheaf.  
\end{subcor}
\begin{proof}Clear from Lemma~\ref{sublemma:basics}~(2) and Corollary \ref{subcor:uniqueness of a full sheaf}.
\end{proof}

\subsection{The minimal resolution $Y$ 
of the quotient $X=U/G$}
\label{subsec:minimal resol}
Let $k$ be any algebraically closed field 
of any characteristic, 
and $G$ a finite small subgroup 
of $\GL(2,k)$, which acts 
on $\bA^2_k$ from the left. 
We assume 
throughout this article that the order $|G|$ of $G$ and 
the characteristic of $k$ are coprime. \par

Let $S=k[x,y]$ be the polynomial ring of 
two variables, and $R=S^G$ the subring of $S$ 
consisting of all $G$-invariants.  
Let $U:=\bA^2_k=\Spec S$. $X=U/G=\Spec R$ and 
let $\pi:U\to X$ be 
the natural morphism. Since $G$ is small, 
the surface $X$ has a unique singular point $0$, 
which is a rational singularity \cite{Viehweg97}.
Let $f:Y\to X$ be the minimal 
resolution of $X=U/G$. Thus we have a commutative diagram:
\begin{equation}
\label{diag:comm diag for Y times U}
\begin{diagram}
Y\times_k U&\rTo^{\quad\quad\pi_U}&U\\
\dTo^{\pi_Y}&&\dTo^{\pi}\\
Y&\rTo^{f\quad\ \,}&X=U/G
\end{diagram}
\end{equation}

Let $F$ be {\em the fundamental divisor} of the singularity $(X,0)$. 
That is, 
the minimum of effective divisors $D$ 
of  $Y$  
 such that 
$D\neq 0$, $\Supp(D)\subset f^{-1}(0)$, and $DE'\leq 0$ for any 
irreducible component $E'$ of $f^{-1}(0)$. 

The following is due to \cite{Wunram88}.  
\begin{thm}
\label{thm:special mckay corresp}
Let $G$ be a finite small subgroup of $\GL(2,k)$ and $X=U/G$.
Under the same notation as in Subsection~\ref{subsec:minimal resol},  
\begin{enumerate}
\item there is a bijective correspondence 
between the following sets
\begin{enumerate}
\item[(i)] the set of irreducible components $E_i$ of $E:=f^{-1}(0)$,
\item[(ii)] the set of non-trivial indecomposable full $O_Y$-modules $\cM_i$, 
{\it special} in the sense that $H^1(Y,\cM_i^{\vee})=0$, 
\end{enumerate}
{where the} correspondence 
$\cM_i\mapsto E_i$ 
is given by
\begin{gather*}
c_1(\cM_i)E_j=\delta_{ij},
\end{gather*}
\item the rank of $\cM_i$ 
is equal to $c_1(\cM_i)F$, the multiplicity of $E_i$ in $F$.
\end{enumerate}
\end{thm}

\subsection{$G$-equivariant locally 
free $O_U$-modules and reflexive 
$O_X$-modules}

Let $G$ be a small subgroup of $\GL(2,k)$.
Any $G$-equivariant 
locally free $O_U$-module $\cF$ of finite rank  
is of the form $\cF=\widetilde{M}$ 
for a projective finite $S$-module $M$ with $G$-action. 
Note that $M$ is a free $S$-module 
\footnote{$M$ is a projective $S$-module if 
and only if $M$ is a free $S$-module.} by \cite{Quillen76} and \cite{Suslin76}.

Let $M^G$ be the submodule consisting of all 
$G$-invariants of $M$. Then the natural exact sequence 
$0\to M^G\to M$ of $R$-modules 
 splits because the characteristic 
of $k$ is prime to the order $|G|$. In other words, 
there is an $R$-submodule $N$ of $M$ such 
that $M\simeq M^G\oplus N$.

For a $G$-equivariant locally free $O_U$-module $\cF$, 
we define a {subsheaf} $\cF^G$ of $\pi_*(\cF)$ 
consisting of $G$-invariant sections by 
$$\cF^G(W)=\Gamma(\pi^{-1}(W),\cF)^G=\Gamma(W, \pi_*(\cF))^G 
$$for any open subscheme $W$ of $X$.  
Since $\pi$ is finite, 
$\cF^G$ 
{is a coherent $O_X$-submodule of $\pi_*\cF$}  
associated to the $R$-module $M^G$. 
See \cite[\S~4]{BKR01}
for more about $G$-sheaves.  
Let $U'=U\setminus\{0\}$, $X'=X\setminus\{\pi(0)\}$, $i:U'\hookrightarrow U$ 
and $j:X'\hookrightarrow X$ the natural inclusions. \par

\begin{subrem}\label{subrem:equiv of cF and cF^G}
If $\cF$ is a $G$-equivariant locally free $O_U$-module,\footnote{$\cF$ is a free $O_U$-module by \cite{Quillen76} and \cite{Suslin76}.} 
$\cF^G$ is a reflexive $O_X$-module. Conversely  
if {$\cF=(\pi^*(\cG))^{\vee\vee}$} for a reflexive $O_X$-module $\cG$, then 
$\cF$ is a $G$-equivariant locally free $O_U$-module such that 
$\cF^G\simeq\cG$.  
\end{subrem}

\begin{subdefn}\label{subdefn:categories of sheaves}
For a normal surface $Z$ over $k$, we denote by  
$\Coh(Z)$ (resp. $\LF(Z)$, 
$\Rfl(Z)$) the category of coherent $O_Z$-modules, (resp. 
the category of locally free $O_Z$-modules of finite rank, 
the category of coherent reflexive $O_Z$-modules). 
If $Z$ has a $G$-action, then we denote by $\Coh^G(Z)$ (resp. $\LF^G(Z)$,  
$\Rfl^G(Z)$) the category of $G$-equivariant coherent $O_Z$-modules, (resp. 
the category of $G$-equivariant locally free $O_Z$-modules of finite rank, 
the category of $G$-equivariant coherent reflexive $O_Z$-modules). 
Let $G\op{-}\Hom_{O_Z}(A,B)$ (resp. 
$\Hom_{O_Z}(A,B)$) 
be the set of all $G$-homomorphisms (resp. all homomorphisms) 
from $A$ to $B$ for $A,B\in\Coh^G(Z)$ (resp. $A,B\in\Coh(Z)$). 
\end{subdefn}

\begin{sublemma}\label{lemma:equivalence}
The functor $\Ginv:\cF\mapsto \cF^G$ is 
an equivalence of the categories $\LF^G(U)$ and 
 $\Rfl(X)$.
\end{sublemma}
\begin{proof}
Since $\pi':U'\to X'$ is flat surjective, 
by faithfully flat descent 
\cite[Expos\'e VIII, Theorem 1.1]{SGA1}, 
$\Coh^G(U')$ and $\Coh(X')$ 
are equivalent 
by the functor $$\cF\mapsto \text{the descent of $\cF$ to $X'$.}$$ 
If $\cF=(\pi')^*\cH$ for a 
coherent $O_{X'}$-module $\cH$, 
then $\cF^G=(\pi^*\cH)^G=\cH$. In other words,  
$\cF^G$ is the descent of $\cF$ to $X'$, 
that is, $\cF=(\pi')^*(\cF^G)$. Moreover the category equivalence 
of $\Coh^G(U')$ and $\Coh(X')$ implies
$$G\op{-}\Hom_{O_{U'}}(\cF,\cG)=\Hom_{O_{X'}}(\cF^G,\cG^G).
$$\par
By Remark~\ref{subrem:equiv of cF and cF^G}, 
we see that $\Ginv:\cF\mapsto \cF^G$ is an equivalence 
between $\LF^G(U')$ and $\LF(X')$. 
By Lemma~\ref{sublemma:reflexive}~(2), $\LF(X')=\Rfl(X')$. 
By Lemma \ref{sublemma:reflexive}~(2)-(4), 
the open immersion $i:U'\hookrightarrow U$ 
 induces an equivalence between $\LF^G(U')$ and $\LF^G(U)$. 
Similarly $j:X'\hookrightarrow X$ induces an equivalence 
between $\Rfl(X')$ and $\Rfl(X)$.\par
Summarizing all the above, we see that 
$\Ginv:\cF\mapsto \cF^G$ is an equivalence 
between $\LF^G(U)$ and $\Rfl(X)$. See also \cite[Proposition~2.2]{Auslander86}.
\end{proof}

\begin{subdefn}For an irreducible representation $\rho$ of $G$, 
we define 
\begin{align*}
M_\rho:&=(O_U \otimesk  \rho^*)^G:=[\pi_*(O_U)\otimesk  \rho^*]^G,\\
\cM_{\rho}:&{=f^*M/O_Y\text{-torsions}}.
\end{align*}
\end{subdefn}

\begin{subcor}\label{cor:dual}
Let $\rho, \sigma$ be irreducible representations of $G$, and 
$\rho^*$ the dual representation of $\rho$,  Then 
$M_\rho$ is a reflexive $O_X$-module and 
$${\chom}_{O_X}(M_\rho, M_\sigma) \simeq M_{\rho^* \otimesk  \sigma}.$$ 
Especially, the dual $O_X$-module $M_\rho^\vee$ of $M_{\rho}$ 
is isomorphic to $M_{\rho^*}$.
\end{subcor}

{
\begin{subdefn}
\label{subdefn:special repres}
An irreducible representation $\rho$ of $G$ is said to be {\it special} if 
its corresponding full $O_Y$-module $\cM:=\cM_{\rho}:=[p_*q^*(O_U\otimesk \rho^*)]^G$ is special, that is, $H^1(Y,\cM^{\vee})=0$.  
A $k[G]$-module $N$ is said to be {\em special} if $N$ is an irreducible $k[G]$-module isomorphic 
to a special representation. 
\end{subdefn}
}

{
We show the existence of non-special representations when $G \not\subset \SL(2,k)$.
}

\begin{sublemma}\label{sublem:special_K}
For a full sheaf $\cM$ on $Y$,
$\cM$ is special if and only if $\cM \otimes_{O_Y} K_Y$ is a full sheaf.
\end{sublemma}
\begin{proof}
Since $\cM$ and $K_Y$ are generated by global sections 
by Corollary \ref{cor:KY is full},  
so is $\cM \otimes_{O_Y} K_Y$.
Lemma follows from $H^1((\cM \otimes_{O_Y} K_Y)^\vee 
\otimes K_Y)=H^1(\cM^\vee)$.
\end{proof}
\begin{sublemma}
\label{lemma:G not in SL}
If $G$ is not contained in $\SL(2, k)$, 
then the following is true.
\begin{enumerate}
\item If $\cM$ is an indecomposable special full sheaf such {that} $\cM \otimes_{O_Y} K_Y$ is special, then $\cM\simeq O_Y$, $K_Y$ is special, 
$K_YF=1$ and $F^2=-3$.
\item {If $\rho\in\Irr(G)$ is special with  
$\rho\otimesk\det\rho_{\nat}^{\vee}$ special, then $\rho\simeq\rho_0$, 
$\det\rho_{\nat}^\vee$ is special, $K_YF=1$ and $F^2=-3$.}
\end{enumerate}
\end{sublemma}
\begin{proof}
The assumption $G \not\subset \SL(2, k)$ implies that $K_YF>0$, 
because otherwise, any irreducible component $E$ of $F$ 
is a smooth rational curve with $E^2=-2$, hence 
the singularity $U/G$ is one of ADE, hence $G\subset\SL(2, k)$. \par 
We shall prove (1). Assume that $\cM \not\simeq O_Y$, and $\cM$ is 
an indecomposable special full sheaf of rank $r$ such that 
$\cM \otimes_{O_Y} K_Y$ is special. 
Then $r=c_1(\cM)F$ and $r=c_1(\cM \otimes_{O_Y} K_Y)F$
by Theorem \ref{thm:special mckay corresp}~(2). This contradicts 
$K_YF>0$.  It follows that $\cM\simeq O_Y$ and 
 $K_Y$ is a special full sheaf of rank $1$. Note that 
$K_Y$ is always a full sheaf by Corollary \ref{cor:KY is full}. 
We also note that $K_Y$ is special if and only if $K_Y F=1$ 
by Theorem \ref{thm:special mckay corresp}~(2), which is equivalent to $F^2=-3$ 
because $H^0(O_F)=k$ and $H^1(O_F)=0$ by \cite[Theorem 3]{Artin66}.
For (2), let $\cM=\cM_{\rho}$. Then 
$\cM \otimes_{O_Y}K_Y\simeq \cM_{\rho\otimesk\det\rho_{\nat}^{\vee}}$. 
Hence $\cM$ is a special full sheaf such that 
$\cM \otimes_{O_Y}K_Y$ is special. 
Hence (2) follows from (1).
\end{proof}

\begin{subcor}
\label{cor:nonspecial repres}
{Suppose $G$ is not contained in $\SL(2, k)$. Then
for a non-trivial special representation $\rho$, 
$\rho \otimes \det\rho_{\nat}^\vee$
is non-special.
Especially, non-special representations exist.}
\end{subcor}

\subsection{{Universal $G$-cluster and full sheaves}}

The following is due to \cite{Ishii02}. 
\begin{subthm}
Let $U=\bA^2_k$, $G$ a finite subgroup of $\GL(2,k)$ 
such that the order $|G|$ of $G$ is 
prime to the characteristic of $k$.
Then {$Y=\GHilb(U)$} is connected and it is a minimal resolution of $X:=U/G$.
\end{subthm}

\begin{subdefn}Let $\cZ$ be the universal cluster over $Y=\GHilb(U)$.
Consider the commutative diagram:
\begin{equation}\label{diag:comm diag for cZ}
\begin{diagram}
\cZ&\rTo^{\quad q}&U\\
\dTo^{p}&&\dTo^{\pi}\\
Y&\rTo^{f\ \ }&X=U/G
\end{diagram}
\end{equation}
where 
$q:=(\pi_U)_{\cZ}$ and $p:=(\pi_Y)_{\cZ}$. 
Note that 
$$\cZ\times_Y(Y\setminus F)\simeq (U\times_X(X\setminus\{O\}).$$
\end{subdefn}

\begin{sublemma}
\label{sublemma:direct image of OZ and OU}
The following is true:
\begin{enumerate}
\item $q_*(O_{\cZ})=O_U$ and $f_*p_*(O_{\cZ})\simeq
\pi_*q_*(O_{\cZ})=\pi_*(O_U)$, 
\item $\pi_*(O_U)$ is a reflexive $O_X$-module, 
\item  $p_*(O_{\cZ})$ is the torsion free pullback of $\pi_*(O_U)$ by $f$, 
\item for any irreducible representation $\rho$ of $G$, 
$$\cM_{\rho} \cong [p_*q^*(O_U)\otimesk \rho^*]^G=[p_*(O_{\cZ})\otimesk \rho^*]^G,$$
\item $p_*(O_{\cZ})=\bigoplus_{\rho\in\Irr(G)}\cM_{\rho}\otimesk \rho$ and $\pi_*(O_U)=\bigoplus_{\rho\in\Irr(G)}M_{\rho}\otimesk\rho$.
\end{enumerate}
\end{sublemma}
\begin{proof} Let $q':\cZ':=\cZ\setminus q^{-1}(0)\to U'$ (resp. 
$p':\cZ'\to Y'=Y\setminus F$) be the restriction of $q$ (resp. of $p$). 
First we prove $q_*(O_{\cZ})=O_U$. 
Since $q'$ is an isomorphism, 
$(q')_*(O_{\cZ'})\simeq O_{U'}$. 
It follows that $O_U\subset q_*(O_{\cZ})\subset i_*((q')_*(O_{\cZ'}))=i_*(O_{U'})=O_U.$
Hence $q_*(O_{\cZ})=O_U$. This proves (1). \par
Let $\cF=\pi_*(O_U)$. Let $X'=X\setminus\{0\}$ 
and $j:X'\hookrightarrow X$. 
Then 
$$\Gamma(X,j_*\cF)=\Gamma(X',\cF)=\Gamma(U\setminus\{0\},O_U)=\Gamma(U,O_U)=\Gamma(X,\cF).$$ Since $X$ is affine, 
this shows $j_*(\cF)=\cF$. 
By Lemma~\ref{sublemma:reflexive}~(4), $\pi_*(O_U)=\cF$ 
is reflexive.  This proves (2). \par
Next we prove (3). See also the proof of \cite[Theorem 3.1]{Ishii02}. 
{Since $\cZ$ is finite and flat over $Y$,
$p_*(O_{\cZ})$ is locally free.
Moreover, it is generated by global sections because 
it is a quotient of the quasi-coherent $O_Y$-module $O_Y \otimes_k k[x, y]$.}
So 
{it follows from $f_*p_*(O_{\cZ}) \cong \pi_*{O_U}$ that}
$p_*(O_{\cZ})$ is the torsion free pullback of $\pi_*(O_U)$ by $f$. \par
Next we prove (4).  {It is proved in \cite[Corollary 3.2]{Ishii02} that 
$[p_*(O_{\cZ})\otimesk \rho^*]^G$ is a full $O_Y$-module. We have 
\begin{align*}
f_*([p_*(O_{\cZ})\otimesk \rho^*]^G)
&=[f_*p_*(O_{\cZ})\otimesk \rho^*]^G
=[\pi_*q_*(O_{\cZ})\otimesk \rho^*]^G\\
&=[\pi_*(O_U)\otimesk \rho^*]^G=M_{\rho}.
\end{align*}
Taking the torsion-free pullbacks of the both sides, we obtain (4).}

Next we prove (5). By \cite[Subsection~4.2]{BKR01} and 
since the characteristic of $k$ is prime to $|G|$,  there is a decomposition  
$p_*(O_{\cZ})=\bigoplus_{\rho\in\Irr(G)}\cN_{\rho}\otimesk \rho$ 
for some $O_Y$-module $\cN_{\rho}$.
For any irreducible representation $\gamma$ of $G$, we have 
\begin{align*}
\cM_{\gamma}:&=[p_*(O_{\cZ})\otimesk \gamma^*]^G
\simeq \bigoplus_{\rho\in\Irr(G)}\cN_{\rho}\otimesk (\Hom_k(\gamma,\rho))^G
\simeq \cN_{\gamma}.
\end{align*}
The rest of (5) follows from (1) and (4).  
\end{proof}

\section{Global McKay correspondence}
\label{sec:global McKay}
The purpose of this section is to
prove Theorem~\ref{thm:mckay isom} together with Corollary~\ref{cor:global local isom}, which reformulate Theorem \ref{thm:thm} in a more precise way.
This is a global version of Theorem \ref{thm:local McKay} proved in \cite{Ishii02}
in the sense that it describes a family of $k[G]$-modules in Theorem \ref{thm:local McKay}
parameterized by the exceptional set.
\subsection{Local McKay correspondence}
\label{subsec:local McKay}
We first recall results from \cite{Ishii02}.
\begin{defn}
\label{defn:Ishii's Psi and Phi}
We define two functors 
$$\Psi:D^G(U)\to D(Y),\quad \Phi:D(Y)\to D^G(U)$$
by
\begin{align*}
\Psi(A)&=[p_*\bL q^*(A)]^G=[\bR (\pi_Y)_*
(O_{\cZ}\overset{\bL}{\otimes}_{O_{Y\times_k U}}\pi_U^*(A)]^G,\\
\Phi(B)&=\bR (\pi_U)_*
(O_{\cZ}^{\vee}
\overset{\bL}{\otimes}_{O_{Y\times_k U}}\pi_Y^*(B\otimesk \rho_0)
\overset{\bL}{\otimes}_{O_{Y\times_k U}}\pi_U^*K_U[2]).
\end{align*}
where $O_{\cZ}^{\vee}$ is the derived dual of $O_{\cZ}$, that is, 
$\bR{\chom}_{O_{Y\times_k U}}(O_{\cZ}, O_{Y\times_k U})$.
\end{defn}

\begin{thm}\label{thm:Ishii fully faithful}
$\Phi$ is fully faithful and 
$\Psi$ is a left adjoint of $\Phi$.
\end{thm}

See \cite[\S~6]{Ishii02} {and \cite[Proposition 1.1, Lemma 2.9]{Ishii-Ueda15}.}

\begin{subdefn}
\label{subdefn:number_special}
We number 
all irreducible representations $\rho$ in the following manner:
 $\rho_0$ is trivial, $\rho_i$ $(1\leq i\leq m)$ is non-trivial special and 
$\rho_j$ $(j\geq m+1)$ is non-trivial nonspecial. \end{subdefn}
We denote $\cM_{\rho_i}$ by $\cM_i$.

\begin{sublemma}
\label{sublemma:Psi O0rhoi*}
Let $\rho_i$ be an irreducible special representation of $G$, 
$\cM_i=\cM_{\rho_i}$, and $E_i$ the irreducible component of $F$ 
with $c_1(\cM_i)E_j=\delta_{ij}$ 
as in Theorem \ref{thm:special mckay corresp}. 
 Then 
\begin{equation*}
\Psi(O_0\otimesk \rho^*)=
\begin{cases}O_{E_i}(-1)[1]&\text{if $\rho=\rho_i$ is non-trivial special}\\
O_F&\text{if $i=0$, that is, $\rho=\rho_0$ is trivial},\\
0&\text{if $\rho$ is non-trivial, nonspecial}.
\end{cases}
\end{equation*}
\end{sublemma}

See \cite[\S~5]{Ishii02} for Lemma~\ref{sublemma:Psi O0rhoi*}.
\begin{lemma}
\label{lemma:Extk(OZy,O0rhoi)} 
Under the same notation as above,  
\begin{align*}
G\text{-}&\Ext^k_{O_U}(O_{Z_y},(O_0\otimesk \rho_i)):=
\Hom^k_{D^G_c(O_U)}(O_{Z_y},(O_0\otimesk \rho_i))
\\
&=
\begin{cases}
\Ext^k_{O_Y}(O_{F},O_y)&(i=0)\\
\Ext^{k-1}_{O_Y}(O_{E_i}(-1),O_y)&(1\leq i\leq m)
\\
0&(i\geq m+1) 
\end{cases}\\
&=\begin{cases}O_y&(i=0, k=0,1, y\in F),\\
O_y&(1\leq i\leq m, k=1,2, y\in E_i),\\
0&(\text{otherwise}).
\end{cases}
\end{align*}
\end{lemma}

See \cite[\S~7]{Ishii02} for the above lemma.
See also Definition \ref{subdefn:number_special}.

\begin{thm}
\label{thm:local McKay}(Local McKay correspondence)\quad 
Let $\fm$ be the maximal ideal of $O_U$ at the origin, 
$y\in Y$ and $Z_y$ the $G$-invariant cluster of $U$ corresponding to $y$ and 
$I_{Z_y}$ the ideal of $O_U$ defining $Z_y$. Then 
the $k[G]$-module $\Gen(I_y):=I_{Z_y}/\fm I_{Z_y}$ is given by
\begin{align*}
&
\begin{cases}
\rho_i\oplus \rho_0&\text{if}\ y\in E_i\setminus \cup_{j\neq i}E_j\\
\rho_i\oplus\rho_j\oplus\rho_0&\text{if}\ y\in E_i\cap E_j,\ i\neq j.
\end{cases}
\end{align*}
\end{thm}

See \cite[Theorem 7.1]{Ishii02}.

\subsection{The ideals $n_Y$ and $I_Y$}\label{subsec:nX}
Let $U=\bA^2_k$, $X=U/G$, $Y=\GHilb(U)$ and 
$f:Y\to X$ the natural morphism, which is the minimal resolution of $X$. 
Let $\pi_Y:Y\times_kU\to Y$ and $\pi_U:Y\times_kU\to U$ 
be the first and the second projection. 
Since $Y\times_{X}X\simeq Y$, {we see that}
$Y$ is a closed subscheme of $Y\times_k X$. 
Let $I_Y$ be the ideal   of 
$O_{Y\times_{k}X}$ defining $Y$, and $n_Y
=I_YO_{Y\times_kU}$. \par
Let $\cZ$ be the universal subscheme of 
$Y\times_{k} U$, and $\cI$ the ideal 
sheaf of $O_{Y\times_{k} U}$ 
defining $\cZ$. 
 We have a commutative diagram of structure sheaves:
\begin{diagram}
&O_{\cZ}&\lTo^{q^*}& O_U \\
&\uTo^{p^*}&&\uTo^{\pi^*}& \\
&O_Y&\lTo^{f^*}&O_X.
\end{diagram}where 
\begin{gather*}
O_Y\simeq O_Y\otimes_{O_X}O_X\simeq O_Y\otimesk  O_X/I_Y,\\
O_Y\otimes_{k}O_X\overset{\id_Y^*\otimes\pi^*}{\to} O_Y\otimes_{k}O_U\overset{
p^*\otimes q^*}{\to} O_{\cZ},\\
O_Y\otimes_{k}O_X/I_Y\simeq O_Y\overset{p^*}{\to} O_{\cZ}=O_Y\otimesk  O_U/\cI.
\end{gather*}

\subsection{The fundamental divisor $F$}
\label{subsec:fund div}
Let $F$ 
be the fundamental divisor of the singularity $(X,0)$ 
(see Subsection~\ref{subsec:minimal resol}).

\begin{sublemma}
\label{sublemma:ideal nY subset cI}
Let $\fm$ be the (maximal) ideal  of $O_{U}$ defining the origin $0$, and  
$I_F$ the ideal of $O_{Y}$ defining the fundamental divisor $F$.
Then
\begin{enumerate}
\item $O_{Y\times_XU}\simeq O_Y\otimesk O_U/\fn_Y$, that is, 
$\fn_Y$ is the ideal of $O_Y\otimesk O_U$ 
defining the fiber product $Y\times_XU$. 
\item $\cZ\subset Y\times_XU$, that is, $\fn_Y\subset \cI$,
\item $\pi_U^*\fm+\fn_Y=\pi_U^*\fm+I_FO_{Y\times_k U}
=\pi_U^*\fm+\cI$,
\item $\fm\cI+\fn_Y\cI=\fm\cI+\cI^2=\fm\cI+I_F\cI$. 
\end{enumerate}
\end{sublemma}
\begin{proof} Let $\fm_X$ be the (maximal) ideal of $O_X$ 
defining the singular point  $0$. 
The ideal $I_Y$ of $O_Y\otimesk O_X$ is 
generated by $(f^*a)\otimes 1-1\otimes a$  for $a\in \fm_X$.
Therefore 
$\fn_Y$ is  generated by 
$(f^*a)\otimes 1-1\otimes(\pi^*a)$ for $a\in \fm_X$. 
Hence $O_Y\otimesk O_U/\fn_Y$ is the structure sheaf of the fiber product 
$Y\times_XU$.  Since $\cZ\subset Y\times_XU$, we have 
$I_Y\subset\cI$, so that 
$\fn_Y\subset\cI$. This proves part (2).
\par
The inclusion $\pi_U^*\fm+\fn_Y\subset \pi_U^*\fm+\cI$ is clear.  
We shall prove the converse.  Since $\fn_Y$ is generated by 
$(f^*a)\otimes 1-1\otimes(\pi^*a)$ for $a\in \fm_X$, we have 
$(f^*a)\otimes 1\in \pi_U^*\fm+\fn_Y$ because $\pi^*a\in \fm$. 
Thus $\pi_U^*\fm+\fn_Y=\pi_U^*\fm+I_FO_{Y\times U}$. 
We see
\begin{align*}
O_{Y\times_k U}/(\pi_U^*\fm+I_FO_{Y\times_k U})&
\simeq (O_{Y}/I_F)\otimes_{O_{Y\times_k U}} (O_U/\fm)=O_F,\\
O_{Y\times_k U}/(\pi_U^*\fm+\cI)&\simeq 
(O_{Y\times_k U}/\cI)\otimes_{O_{Y\times_k U}} (O_U/\fm)\\
&\simeq O_{\cZ}\otimes_{O_{Y\times_k U}} (O_U/\fm)=O_F
\end{align*} 
by Lemma~\ref{sublemma:Psi O0rhoi*}~(2). 
It follows that 
$$\pi_U^*\fm+I_FO_{Y\times U}=\pi_U^*\fm+\cI.
$$ 

This proves Lemma.
\end{proof}

\begin{sublemma}
\label{sublemma:ideal rel}Let $\fn:=(\fm\cap k[x,y]^G)O_U$. 
Then the following is true:
\begin{enumerate}
\item $I_{Z_y}\simeq\cI/\fm_y\cI\simeq
\left(\cI/\cI\cap (\fm_y\otimesk O_U)\right)$,
\item $\cI\cap (\fm_y\otimesk O_U)=\fm_y\cI$,
\item $\fn_Y+\fm_y\otimesk O_U=O_Y\otimesk \fn+\fm_y\otimesk O_U$ if $y\in F$, 
\item $(\fn_Y+\fm_y\cI)/\fm_y\cI\simeq \fn$ if $y\in F$.
\end{enumerate}
\end{sublemma}
\begin{proof}
Since $O_{\cZ}$ is $O_Y$-flat, 
\begin{align*}
O_{Z_y}&\simeq O_{\cZ}\otimes_{O_Y}O_y=(O_{Y\times_k U}/\cI)\otimes_{O_Y}O_y\\
&\simeq (O_Y\times_kO_U)/(\cI+\fm_y\otimesk O_U)\\
&\simeq (O_y\times_kO_U)/\left(\cI/\cI\cap (\fm_y\otimesk O_U)\right),
\end{align*}whence 
$I_{Z_y}\simeq \cI/\cI\cap (\fm_y\otimesk O_U)$. 
This proves the part (1).

Since $O_{\cZ}$ is $O_Y$-flat again, the following is an exact 
sequence:
$$0\to \cI\otimes_{O_{Y\times_k U}}O_y\to O_y\otimesk O_U\ (\simeq O_U)\to 
O_{Z_y}(\simeq O_{\cZ}\otimes_{O_{Y\times_k U}} O_y)\to 0,
$$so that we have $I_{Z_y}=\cI/\fm_y\cI$. Hence we have
\begin{align*}
\cI\cap(\fm_y\otimesk O_U)=\fm_y\cI,
\end{align*}which proves the part (2).
Since $\fn_Y$ is generated by $f^*a\otimes 1-1\otimes\pi^*a$ 
for $a\in \fm_X$ and 
$\pi^*\fm_X=\fn$,  
we have $\fn_Y+\fm_y\otimesk O_U=O_Y\otimesk \fn+\fm_y\otimesk O_U$ 
for $y\in F$, which is the part (3). 
It follows that  
\begin{align*}
(\fn_Y+\fm_y\cI)/\fm_y\cI&\simeq \fn_Y/\fn_Y\cap\cI\cap (\fm_y\otimesk O_U)\\
&= \fn_Y/\fn_Y\cap(\fm_y\otimesk O_U)\quad\text{by Lemma~\ref{sublemma:ideal nY subset cI}~(2)}\\
&\simeq  \fn_Y+\fm_y\otimesk O_U/(\fm_y\otimesk O_U)\\
&\simeq  (O_Y\otimesk \fn+\fm_y\otimesk O_U)/(\fm_y\otimesk O_U)\simeq  \fn.
\end{align*} This completes the proof. 
\end{proof}

\begin{subdefn}We define
\begin{align*}
\cV:&=\cI/(\fm\cI+\fn_Y),\\ 
\cV^{\dagger}:&=\cI/(\fm\cI+I_F\cI)\simeq 
(\cI/\fm\cI)\otimes_{O_Y}O_F.
\end{align*}
\end{subdefn}

\begin{sublemma}
\label{sublemma:cV and cVdagger} For $y\in F$, 
\begin{gather*}
\cV\otimes_{O_Y} O_y 
\simeq I_{Z_y}/(\fm I_{Z_y}+\fn);\
\cV^{\dagger}\otimes_{O_Y} O_y
\simeq I_{Z_y}/\fm I_{Z_y}.
\end{gather*}
\end{sublemma}
\begin{proof}
Let $y\in Y$ and $\fm_y$ the maximal ideal of $O_{Y,y}$. 
Hence $O_y=O_{Y,y}/\fm_y=O_y$. Since $I_{Z_y}=\cI/\fm_y\cI$, 
we have 
\begin{align*}
\cV\otimes_{O_Y} O_y&\simeq 
(\cI/(\fm\cI+\fn_Y))\otimes_{O_Y} O_y=\cI/(\fm\cI+\fn_Y+\fm_y\cI)\\
&\simeq (\cI/\fm_y\cI)/\left(\fm(\cI/\fm_y\cI)
+(\fn_Y+\fm_y\cI)/\fm_y\cI)\right)\\
&\simeq I_{Z_y}/(\fm I_{Z_y}+\fn),\\
\cV^{\dagger}\otimes_{O_Y} O_y
&=\cI/(\fm\cI+I_F\cI+\fm_y\cI)=\cI/(\fm\cI+\fm_y\cI)\\
&\simeq (\cI/\fm_y\cI)/(\fm\cI+\fm_y\cI)/\fm_y\cI)\\
&\simeq (\cI/\fm_y\cI)/m(\cI/\fm_y\cI)\simeq I_{Z_y}/\fm I_{Z_y}.
\end{align*}
\end{proof}

\begin{defn}
For a coherent $O_{Y\times U}$-module $J$,
we define a functor 
\begin{gather*}
\Psi_J:D_c^G(U)\to D_c(Y)
\end{gather*}
by
\begin{align*}
\Psi_J(A)&=[p_*(\bL \pi_U^*(A)\overset{\bL}{\otimes}_{O_{Y\times_k U}}J)]^G
=\bR (\pi_{Y})_*(\bL\pi_U^*(A)\overset{\bL}
{\otimes}_{O_{Y\times_k U}}J)]^G
\end{align*}where $A\in D_c^G(U)$.  
Note that $\Psi=\Psi_{O_{\cZ}}$.
\end{defn}

\begin{lemma}
\label{lemma:PsicI}
The following is true:
\begin{enumerate}
\item 
$\Psi(O_0\otimesk\rho_i^*)=\begin{cases}
O_{F}&(i=0),\\
O_{E_i}(-1)[1]& (\text{$\rho_i$ : non-trivial special}),
\end{cases}$
\item $\Psi_{O_{Y\times_k U}}(O_0\otimesk\rho_i^*)=
\begin{cases}O_{Y} & (i=0),\\
0&(i\neq 0),
\end{cases}$
\item $\Psi_{\cI}(O_0\otimesk\rho_i^*)=
\begin{cases}
O_{Y}(-F)&(i=0) ,\\
O_{E_i}(-1)&(i\neq 0),\\
0&\text{(otherwise).}
\end{cases}$ 
\end{enumerate}
\end{lemma}
\begin{proof}The part (1) follows from Lemma~\ref{sublemma:Psi O0rhoi*}.
Since $O_{Y\times_k U}$ is $O_{Y\times_k U}$-flat, by definition, 
\begin{align*}
\Psi_{O_{Y\times_k U}}(O_0\otimesk\rho_i^*)
&=[p_*(O_{Y}\otimesk (O_0\otimesk\rho_i^*))]^G\\
&=[O_{Y}\otimesk\rho_i^*]^G=
\begin{cases}O_{Y}&(i=0)\\
0&(i\neq 0)
\end{cases}
\end{align*}

The part (3) follows from the parts (1) and (2) and the exact sequence:
\begin{diagram}
&\rTo&\Psi^{-2}_{\cI}(O_0\otimesk \rho_0^*)
&\rTo&\Psi^{-2}_{O_{Y\times_k U}}(O_0\otimesk \rho_0^*)
&\rTo&\Psi^{-2}_{O_{\cZ}}(O_0\otimesk \rho_0^*)\\
&\rTo&\Psi^{-1}_{\cI}(O_0\otimesk \rho_0^*)&
\rTo&\Psi^{-1}_{O_{Y\times_k U}}(O_0\otimesk \rho_0^*)&\rTo
&\Psi^{-1}_{O_{\cZ}}(O_0\otimesk \rho_0^*)\\
&\rTo&\Psi^0_{\cI}(O_0\otimesk \rho_0^*)&\rTo&
\Psi^{0}_{O_{Y\times_k U}}(O_0\otimesk \rho_0^*)&\rTo
&\Psi^{0}_{O_{\cZ}}(O_0\otimesk \rho_0^*)&\rTo&0.
\end{diagram}

This proves the part (3).
\end{proof}

The following is a {global version (global over the exceptional set)} of Theorem \ref{thm:local McKay}:  
\begin{thm}\label{thm:mckay isom}
{There are isomorphisms}
\begin{gather*}
\cV\simeq 
\sum_{\rho_i\neq\rho_0}O_{E_i}(-1)\otimesk \rho_i,\quad  
\cV^{\dagger}\simeq \cV\oplus O_F(-F)\otimesk \rho_0,
\end{gather*}
where $\rho_i$ 
ranges over all non-trivial special irreducible representations 
of $G$. 
\end{thm}

\begin{cor}
\label{cor:global local isom} Let $\fn=(\fm\cap k[x,y]^G)O_U$. Then
{the fibers of $\cV$ and $\cV^{\dagger}$ over $y \in F$ are}
\begin{align*}
\cV\otimes_{O_Y} O_y&=
\begin{cases}
\rho_i&(y\in E_i\setminus \cup_{j\neq i}E_j)\\
\rho_i\oplus\rho_j&(y\in E_i\cap E_j,\ i\neq j), 
\end{cases}\\
\cV^{\dagger}\otimes_{O_Y} O_y&=
\begin{cases}
\rho_i\oplus \rho_0&(y\in E_i\setminus \cup_{j\neq i}E_j)\\
\rho_i\oplus\rho_j\oplus\rho_0&(y\in E_i\cap E_j,\ i\neq j).
\end{cases}
\end{align*}
\end{cor}

We note $E_i=C(\rho_i)$ and 
$\Gen(I_{Z_y})=I_{Z_y}/\fm I_{Z_y}$ in Theorem \ref{thm:thm}.

\begin{proof}[Proof of Theorem~\ref{thm:mckay isom} and  
Corollary~\ref{cor:global local isom}]

By Lemma~\ref{lemma:PsicI}, we have 
\begin{align*}
\cI/\fm \cI&
=p_*(\bL \pi_U^*(O_0)\overset{\bL}{\otimes}_{O_{Y\times_k U}}\cI)
\simeq\sum_{\rho:\text{irred.}}\Psi_{\cI}(O_0\otimesk \rho^*)\otimes\rho\\
&\simeq 
\bigoplus_{i=1}^nO_{E_i}(-1)\otimesk \rho_i\bigoplus O_Y(-F)\otimesk \rho_0,
\end{align*}which is an isomorphism in $D_c(Y)$. However since 
the rhs is concentrated to degree zero only, it is an isomorphism of 
$O_Y$-modules. It follows 
\begin{align*}
\cV^{\dagger}&=\cI/(\fm\cI+I_F\cI)
\simeq(\cI/\fm \cI)\otimes_{O_Y}O_F\\
&\simeq 
\bigoplus_{i=1}^nO_{E_i}(-1)\otimesk \rho_i\bigoplus O_F(-F)\otimesk \rho_0.
\end{align*}

It remains to compute $\cV$.  
By Lemma~\ref{sublemma:cV and cVdagger}, 
\begin{align*}
I_{Z_y}/\fm I_{Z_y}&=\bigoplus_{i=1}^n O_{E_i}(-1)
\otimes_{O_Y} O_y\otimesk \rho_i\bigoplus O_y(-F)\otimesk \rho_0,\\
I_{Z_y}/(\fm I_{Z_y}+\fn)
&\simeq\bigoplus_{i=1}^n O_{E_i}(-1)\otimes_{O_Y} O_y\otimesk \rho_i
\end{align*}where 
every generator of $O_F(-F)\otimesk \rho_0\subset\cV^{\dagger}$ is 
the image of a $G$-invariant polynomial in $\cI$, which reduces to zero 
in $\cV$ because $\fm\cap k[x,y]^G\subset\fn$. 
Hence we have $\cV=\bigoplus_{i=1}^n O_{E_i}(-1)\otimesk \rho_i.$
This proves Theorem~\ref{thm:mckay isom} and 
Corollary~\ref{cor:global local isom}.
\end{proof}

\begin{rem}
Whether or not one can find an isomorphism 
similar  to Theorem \ref{thm:mckay isom} is of some interest when $U=\bC^n$ and 
$G$ is a very natural subgroup of $\GL(U)$. For example, 
$U$ is the Griess algebra or the vertex operator algebra for the big Monster $\bM$. For $G=M_{24}$, the Mathieu group of degree 24, $U$ is the Leech lattice. 
\end{rem}

\subsection{Integral functors with dual kernel objects}
In this subsection, we slightly alter the functors $\Phi$ and $\Psi$
into functors which are more suitable for our purpose.

For $J \in D^G(Y \times U)$, we define a functor
$\Phi_J: D_c(Y) \to D_c^G(U)$
by $$\Phi_J(-)=\bR(\pi_U)_*(\pi_Y^*(-)
\overset{\bL}{\otimes}_{O_{Y\times_k U}}J).$$
In the sequel, we assume that the support of $J$ is proper over both $Y$ and $U$.
Then $\Phi_J$ is extended to
$$\Phi_J: D(Y) \to D^G(U).$$
We denote the right adjoint of $\Phi_J$ by $\Phi_J^*$. 
\begin{sublemma}
\label{lemma:PhiJdual}
Let $J$ be an object of $D^G(Y \times_k U)$ {whose support is proper over both $Y$ and $U$}, 
$J^\vee$ the derived dual of $J$,  $A\in D(Y)$ 
and $A^\vee$ the derived dual of $A$.   Then 
\begin{align*}
\Phi_{J^\vee}(A)&\simeq\Phi_J(A^{\vee} 
\overset{\bL}{\otimes}_{O_Y} K_Y[2])^{\vee}.
\end{align*}
\end{sublemma}
\begin{proof}
The assertion follows from:
\begin{align*}
\Phi_{J^\vee}(A)^\vee & \simeq \Rhom_{O_U}(\bR(\pi_U)_*(\pi_Y^*(A) \overset{\bL}{\otimes}_{O_{Y\times_k U}} J^\vee),\, O_U) \\
& \simeq \bR(\pi_U)_* (\Rhom_{O_{Y\times_k U}}(\pi_Y^*(A)\overset{\bL}{\otimes}_{O_{Y\times_k U}} J^\vee,\, \pi_Y^*K_Y[2])) \\
& \simeq \bR(\pi_U)_* (\pi_Y^*(A^{\vee}\overset{\bL}{\otimes}_{O_Y} K_Y[2]) \overset{\bL}{\otimes}_{O_{Y\times_k U}} J) \\
& \simeq \Phi_J(A^{\vee} \overset{\bL}{\otimes}_{O_Y} K_Y[2]).
\end{align*} \end{proof}

See also \cite[Remark~5.8]{Huybrechts06}.\par

\begin{subcor}\label{cor:Phi and PhiOZ}
$\Phi_{O_\cZ}(B)\simeq
\Phi(B^{\vee}\otimes_{O_Y} K_Y)^{\vee}\otimes_{O_{Y\times_k U}} \pi_U^*K_U$.  
\end{subcor}

\begin{subcor}
\label{cor:rightadj of PhiOZ}
Let $\phi$ and $\psi$ be the same as in Definition~\ref{defn:Ishii's Psi and Phi}. 
Then the right adjoint $\Phi_{O_{\cZ}}^*$ of $\Phi_{O_{\cZ}}$ is given by
\begin{align*}
\Phi_{O_{\cZ}}^*(B)&=\Psi(B^{\vee} \overset{\bL}{\otimes}_{O_U} K_U)^{\vee}
\overset{\bL}{\otimes}_{O_Y} K_Y
\end{align*} for $A\in D(Y)$ and $B\in D^G_c(U)$.
\end{subcor}
\begin{proof}The assertion follows from:
\begin{align*}
&\Hom_{D^G(U)}(\Phi_{O_\cZ}(A), B)
\simeq \Hom_{D^G(U)}(B^\vee, \Phi_{O_\cZ}(A)^\vee) \\
& \simeq \Hom_{D^G(U)}(B^\vee \overset{\bL}{\otimes}_{O_U} K_U, 
\Phi_{O_\cZ}(A)^\vee 
\overset{\bL}{\otimes}_{O_U} K_U) \\
& \simeq \Hom_{D^G(U)}(B^\vee \overset{\bL}{\otimes}_{O_U} K_U, 
\Phi(A^\vee \overset{\bL}{\otimes}_{O_Y} K_Y))\quad(\because \text{by Corollary \ref{cor:Phi and PhiOZ}}) \\
& \simeq \Hom_{D(Y)}(\Psi(B^\vee \overset{\bL}{\otimes}_{O_U} K_U), 
A^\vee \overset{\bL}{\otimes}_{O_Y} K_Y)
\quad(\because \text{by Theorem \ref{thm:Ishii fully faithful}}) \\
& \simeq \Hom_{D(Y)}(A, \Psi(B^\vee \Lotimes_{O_U} K_U)^\vee
\Lotimes_{O_Y} K_Y).
\end{align*}
\end{proof}

Now Lemma~\ref{sublemma:Psi O0rhoi*} (see also \cite[Theorem 5.1]{Ishii02})  
is restated as follows:
\begin{subcor}\label{cor:Phi(O0rho)}
For an irreducible representation $\rho$ of $G$,
$$
\Phi_{O_\cZ}^*(O_0 \otimesk  \rho) \simeq
\begin{cases}
O_{E(\rho)}(-1) & \text{if $\rho$ is non-trivial special,} \\
\omega_F[1] & \text{if $\rho=\rho_0$,} \\
0 & \text{otherwise.}
\end{cases}
$$
\end{subcor}
Here, for a special representation $\rho=\rho_i$,
$E(\rho)=E_i$ denotes the corresponding irreducible component of $F$.

\section{Extensions of the socle and the cup products}
\label{sec:extensions and cup}

\subsection{Basic construction of extensions}
\label{subsec:basic constr}
\begin{subdefn}
\label{subdefn:pushforward}
Let
\begin{diagram}
(s)&:&0&\rTo& A&\rTo^{\phi} &C&\rTo^{\eta}& F&\rTo& 0.
\end{diagram}be an exact sequence of $R$-modules.
For any $R$-homomorphism $\psi:A\to B$, 
we define {\em the pushforward} $\psi_*(s)$ of $(s)$ by $\psi$ to be 
an exact sequence:  
\begin{diagram}
\psi_*(s)&:&0&\rTo& B&\rTo^{(\id_B,0)} &C_{\psi}&\rTo^{\eta\cdot p_2}& F&\rTo& 0.
\end{diagram}
where $C_{\psi}:=B\times C/(\psi,{-}\phi)(A)$.
\end{subdefn}
\begin{subdefn}
For any $R$-homomorphism $\gamma:E\to F$, we define 
{\em the pullback} $\gamma^*(s)$ of $(s)$ by $\gamma$ to be an exact sequence:
\begin{diagram}
\gamma^*(s)&:&0&\rTo& A&\rTo^{(0,\phi)} &C^{\gamma}&\rTo^{\eta\cdot p_2}& E&\rTo& 0
\end{diagram}
where $C^{\gamma}:=E\times_F C:=\{(e,c)\in E\times C; \gamma(e)=\eta(c)\}$.
\end{subdefn}

\begin{defn}
\label{defn:socle}

For an $O_{U,0}$-module $M$, 
we define the socle $\Soc(M)$ of $M$ 
to be the sum of all minimal $O_{U,0}$-submodules of $M$.
For a cluster $Z_y$ with $y\in F$, $O_{Z_y}=O_U/I_y$ is 
an $O_U$-module supported by the origin of $U$. 
Let $\fm$ be the maximal ideal of 
$O_U$ defining $O$. Then it is easy to see 
$$\Soc(O_{Z_y})=[I_y:\fm]/I_y.$$
\end{defn}

\subsection{The extension $O_U/\fm I_y$}
\label{subsec:socle OU/mIy}
The purpose of this subsection is to study the following exact sequence:
\begin{equation*}
\label{seq:basic exact}
\begin{diagram}
(t_1)&:&0&\rTo&I_y/\fm I_y&\rTo&O_U/\fm I_y&\rTo^{\phi}&O_{Z_y}&\rTo&0\\
&&&&\Vert&&\cup&&\cup&&\\
(t_2)&:&0&\rTo&I_y/\fm I_y&
\rTo&[I_y:\fm]/\fm I_y&\rTo&[I_y:\fm]/I_y&\rTo&0.
\end{diagram}
\end{equation*} 

Let $\Gen(I_y)=I_y/\fm I_y=\bigoplus_{\rho\in \Lambda_y} W(\rho)$ for some 
$G$-invariant $G$-irreducible 
$O_U$-submodule $W(\rho)\neq 0$  where $\Lambda_y$ is the set of irreducible representations appearing
as direct summands in Theorem \ref{thm:local McKay}.

By Lemma~\ref{lemma:Extk(OZy,O0rhoi)}, for every $\rho\in\Lambda_y$, 
there exists a unique ideal $J_{\rho}$ of $O_U$ 
with $\fm I_y\subset J_{\rho}\subset I_y$ and 
a non-trivial extension 
\begin{diagram}
({\ext}_{\rho})\quad&:& 0&\rTo&W(\rho)&\rTo
&O_U/J_{\rho}&\rTo&O_{Z_y}&\rTo&0
\end{diagram}where $I_y/J_{\rho}\simeq W(\rho)$ and 
$\fm I_y=\cap_{\rho\in \Lambda_y} J_{\rho}$ because $\Gen(I_y)=\bigoplus_{\rho\in \Lambda_y} W(\rho)$. \par
We choose and fix $0\neq V(\xi)\subset\Soc(O_{Z_y})$ {with $V(\xi) \cong \xi$}. Since $V(\xi)$ is an $O_U$-submodule of $O_{Z_y}$, there exists an $O_U$-submodule $B(\xi)$ of $[I_y:\fm]$ such that 
$B(\xi)/I_y=V(\xi)$. Then we have  a commutative diagram of exact sequences 
\begin{equation*}
\label{seq:2nd basic exact}
\begin{diagram}
(u_1)&:&0&\rTo&\bigoplus_{\rho\in \Lambda_y} W(\rho)&
\rTo&O_U/\fm I_y&\rTo^{\phi}&O_{Z_y}&\rTo&0,\\
&&&&\Vert&&\cup&&\cup&&\\
(u_2)&:&0&\rTo&\bigoplus_{\rho\in \Lambda_y} W(\rho)&
\rTo&B(\xi)/\fm I_y&\rTo&B(\xi)/I_y&\rTo&0.\\
&&&&\dTo&&\dTo&&\Vert&\\
(u_3)&:&0&\rTo&W(\rho)&
\rTo&B(\xi)/J_{\rho}&\rTo&B(\xi)/I_y&\rTo&0.
\end{diagram}
\end{equation*}

It is easy to see 
\begin{sublemma}\label{lemma:equiv not splits}
Let $V_{\rho}^{\natural}(\xi)$ be any $k[G]$-submodule (not necessarily an $O_U$-submodule) of $B(\xi)/J_{\rho}$ 
 lifting $V(\xi)$. Then the following are equivalent:
\begin{enumerate}
\item the exact sequence $(u_3)$ does not split,
\item no lifting $V^{\natural}_{\rho}(\xi)$ 
of $V(\xi)$ is a $G$-invariant $O_U$-module,
\item $S_1\cdot V^{\natural}_{\rho}(\xi)=W(\rho)$,
\item $S_1\cdot V^{\natural}_{\rho}(\xi)\supset W(\rho)$.
\end{enumerate}
\end{sublemma}

\subsection{Cup-products}
\label{subsec:cup products}

We choose and fix irreducible representations  $\rho$ and $\xi$ of $G$ 
such that $W(\rho)\subset\Gen(I_y)$ and $V(\xi)\subset\Soc(O_{Z_y})$. 

The extension (ext$_{\rho}$) 
is uniquely determined by $\rho$  
in view of Theorem~\ref{thm:local McKay}. 
Since every irreducible 
submodule $V(\xi)$ of $O_{Z_y}$ 
is unique, 
the pullback $(u_3)$ of (ext$_{\rho}$) 
via the inclusion $j:V(\xi)\hookrightarrow  O_{Z_y}$
is written as:
\begin{equation}
\label{seq:pullback to Vxi and Wrho}
\begin{diagram}
0&\rTo&W(\rho)
&\rTo&B(\xi)/J_{\rho}&\rTo&V(\xi)=B(\xi)/I_y&\rTo&0. 
\end{diagram}
\end{equation}  
Associated to (\ref{seq:pullback to Vxi and Wrho}),  
we have a natural cup product 
\begin{equation}
\label{eq:0th cup prod}
\begin{aligned}
\Hom^0_{D^G(U)}(O_0\otimesk \xi, O_{Z_y})&\times 
\Hom^1_{D^G(U)}(O_{Z_y},O_0\otimesk\rho)\\
&\to \Hom^1_{D^G(U)}(O_0\otimesk\xi, O_0\otimesk\rho).
\end{aligned}
\end{equation}

By Lemma~\ref{lemma:equiv not splits}, 
the following are equivalent:
\begin{enumerate}
\item[(d)] 
the cup product (\ref{eq:0th cup prod}) is nonzero,
\item[(e)] the extension (\ref{seq:pullback to Vxi and Wrho}) is non-trivial,
\item[(f)] $S_1\cdot V^{\natural}_{\rho}(\xi)\supset W(\rho)$. 
\end{enumerate}

\subsection{The case where $G\subset \SL(2)$}
\label{subsec:sl2}
In this subsection, we assume $G\subset \SL(2)$. 
Since $G\subset\SL(2)$, 
both $\Psi$ and $\Phi$ are 
equivalences of the categories 
such that $\Psi\Phi\cong\id_{D(Y)}$ and
$\Phi\Psi\cong\id_{D^G(U)}$. This is proved by the same argument as in 
\cite{BKR01} and \cite[Theorem 6.2]{Ishii02}. 
The functors $\Phi_{O_\cZ}$ and its adjoint $\Phi_{O_{\cZ}}^*$ are also equivalences
by Corollary \ref{cor:Phi and PhiOZ} and \ref{cor:rightadj of PhiOZ}
(actually, \cite{BKR01} considers $\Phi_{O_\cZ}$ and $\Phi_{O_{\cZ}}^*$ rather than
$\Phi$ and $\Psi$ in this paper).

In what follows we consider a pair $\xi$ and $\rho$ {in $\Irr(G)$} such that 
$\xi\subset\Soc(O_{Z_y})$ and $\rho\subset\Gen(I_y)$. 
Since $G\subset \SL(2)$, {$\xi$ is special and non-trivial.}
 \par
%

We consider 
the cup product {\eqref{eq:0th cup prod}} in $D^G(U)$
\begin{equation*}
\begin{aligned}
\Hom^0_{D^G(U)}(O_0\otimesk \xi, O_{Z_y})&\times 
\Hom^1_{D^G(U)}(O_{Z_y},O_0\otimesk\rho)\\
&\to \Hom^1_{D^G(U)}(O_0\otimesk\xi, O_0\otimesk\rho).
\end{aligned}
\end{equation*} 

 If $\rho\neq\rho_0$, $\xi\neq\rho_0$, then it 
is translated {by Corollary \ref{cor:Phi(O0rho)}} into the following cup product
\begin{equation}
\label{eq:cup prod DcY}
\begin{aligned}
\Hom_{O_Y}(O_{E(\rho)}(-1), &\, O_y)\times 
\Ext^1_{O_Y}(O_y,  \,O_{E(\xi)}(-1))\\
&\to 
\Ext^1_{O_Y}(O_{E(\rho)}(-1), \,O_{E(\xi)}(-1)).
\end{aligned}
\end{equation} 

If $\rho=\rho_0$ and 
$\xi\neq\rho_0$, then 
(\ref{eq:0th cup prod}) is translated
into
\begin{equation*}
\begin{aligned}
\Hom_{O_Y}(O_{E(\xi)}(-1),\,&O_y)\times 
\Ext^2_{O_Y}(O_y, \,\omega_F)
\\
&\to 
\Ext^2_{O_Y}(O_{E(\xi)}(-1),\, \omega_{F})
\end{aligned}
\end{equation*}
and then, by Serre duality, into the following:
\begin{equation}
\label{eq:2nd cup prod DcY}
\begin{aligned}
\Hom_{O_Y}(\omega_F, O_{E(\xi)}&(-1)) \times
\Hom_{O_Y}(O_{E(\xi)}(-1), O_y)
\\
&\to 
\Hom_{O_Y}(\omega_F, O_y).
\end{aligned}
\end{equation}

In order to calculate these cup products, we use the following
lemmas which holds for the minimal resolutions of rational surface singularities:
\begin{lemma}
\label{lemma:cohom Ext Oy OE}Suppose $G\subset\GL(2)$. 
Let $F$ be the fundamental divisor, and 
$C$ any irreducible component of $F$. 
Then we have
\begin{equation*}
\Ext^{q}_{O_Y}(O_y, O_{C}(-1))
\simeq\Ext^{2-q}_{O_Y}(O_{C}(-1), O_y)^{\vee}
=\begin{cases}
k&(q=1,2, y\in C)\\
0&(\text{otherwise})
\end{cases}
\end{equation*}
\end{lemma}
\begin{lemma}
\label{lemma:cup product coh}Suppose $G\subset\GL(2)$. 
Let $F$ be the fundamental divisor, and 
$E,C$ any irreducible component of $F$ with $E\neq C$. 
Then we have
\begin{align*}
\Ext^{q}_{O_Y}(O_{C}(-1),O_{E}(-1))
&=\begin{cases}
k&(q=1, CE=1)\\
0&(\text{otherwise})
\end{cases}\\
\Ext^{q}_{O_Y}(O_{E}(-1),O_{E}(-1))
&=\begin{cases}
k&(q=0)\\
k^{\oplus (1+d)}&(q=2, E^2=-2-d)\\
0&(\text{otherwise})
\end{cases}\\
\Ext^{q}_{O_Y}(O_F, O_E(-1))&
=\begin{cases}
k^{\oplus e}&(q=2, FE=-e<0)\\
0&(q=2, FE=0)\\
0&(\text{otherwise})
\end{cases}
\end{align*}where $d=0$ and 
$e\leq 2$ if $G\subset\SL(2)$ 
\footnote{$e=2$ in the $A_1$ case; otherwise $e\leq 1$. 
}.
\end{lemma}
 The proofs of the lemmas are easy, so we omit them.

\begin{lemma}
\label{lemma:cup prod}Suppose $G\subset\SL(2)$. 
Then the following is true:
\begin{enumerate}
\item the cup product of (\ref{eq:cup prod DcY}) is 
$\begin{cases}
\neq 0&\text{if}\ \{y\}=E(\xi)\cap E(\rho)\\
0&\text{(otherwise)}
\end{cases}$
\item the cup product of (\ref{eq:2nd cup prod DcY}) is 
$\begin{cases}
\neq 0&\text{if}\ FE(\xi)\neq 0, y\in E(\xi)\\
0&\text{(otherwise)}
\end{cases}$
\end{enumerate}
\end{lemma}
\begin{proof}
First we prove (1).  
Let $E=E(\xi)$, $C=E(\rho)$, 
$A=O_{E}(-1)$ and 
$B=O_{C}(-1)$.  Assume $E\neq C$.
It suffices to consider the case $\{y\}=E\cap C$ 
by Lemma~\ref{lemma:cohom Ext Oy OE}. Hence $EC=1$. 
Since $y\in E$, we have an exact sequence
$$(u)\quad :\quad 0\to O_E(-1)\to O_E\to O_y\to 0
$$which is a non-trivial extension given by a nonzero element of 
$\Ext^1_{O_Y}(O_y,A)$. Let $\gamma\in\Hom_{O_Y}(B,O_y)$ 
be a nonzero class.  The cup product (\ref{eq:cup prod DcY}) 
is the pullback 
$\gamma^*(u)$ of $(u)$, which is given explicitly by 
$$0\to O_E(-1)\to O_{E+C}(L)\to O_C(-1)\to 0
$$where $L$ is the unique line bundle of $E+C$ such that $L_E=O_E$ and 
$L_C=O_C(-1)$. This is a non-trivial extension because 
the following is non-trivial:
$$0\to O_E(-1)\to O_{E+C}\to O_C\to 0.
$$
 
If $E=C$, then  $\gamma^*(u)$ is trivial because $\Ext^1_{O_Y}(A,A)=0$
by Lemma~\ref{lemma:cup product coh}.  \par

Next we prove (2).
Since
$\omega_F \cong K_Y \otimes_{O_Y} O_F(-F) \cong O_F(-F)$,
we have
$\Hom_{O_Y}(\omega_F, O_{E(\xi)}(-1)) \cong H^0(O_{E(\xi)}(-FE(\xi)-1))$,
which is non-zero if and only if $FE(\xi) \ne 0$
(note that $FE(\xi) \le 0$ by the definition of the fundamental cycle).
In this case, the evaluation map
$$
\Hom_{O_Y}(\omega_F, O_{E(\xi)}(-1)) \otimesk \omega_F \to O_{E(\xi)}(-1)
$$
is surjective and therefore (2) is proved.
%
%
%
%
\end{proof}

Summarizing 
Subsections~\ref{subsec:socle OU/mIy}--\ref{subsec:sl2}, 
we obtain the following. 
\begin{thm}
\label{thm:mckay tensor prod}
Suppose $G\subset\SL(2)$. Under the above notation,
let $y\in F$, 
$V(\xi)\subset\Soc(I_y)$\ \footnote{hence $\xi\neq\rho_0$ 
by Theorem~\ref{thm:local McKay}} and 
$W(\rho)\subset\Gen(I_y)$.  Then the following are equivalent:
\begin{enumerate}
\item $W(\rho)\subset S_1\cdot V^{\natural}(\xi)$ in $\Gen(I_y)$,
\item the extension (\ref{seq:pullback to Vxi and Wrho}) is non-trivial,
\item the cup product (\ref{eq:0th cup prod}) is nonzero,
\item (\ref{eq:cup prod DcY}) 
is nonzero if $\rho\neq\rho_0$, while (\ref{eq:2nd cup prod DcY}) 
is nonzero if $\rho=\rho_0$, 
\item $E(\xi)E(\rho)=1$ {and $y \in E(\xi) \cap E(\rho)$} if $\rho\neq\rho_0$, while 
$E(\xi)F\neq 0$\ \footnote{this is equivalent to $\xi=\rho_{\nat}$  if it is not the $A_n$-case} {and $y \in E(\xi)$} if $\rho=\rho_0$.
\end{enumerate}
\end{thm}

See also Theorem \ref{thm:cluster-quiver}~(1) for $G\not\subset\SL(2)$.

\section{Semi-orthogonal projections}
\label{sec:semi-orth proj}

\subsection{Semi-orthogonal projection of $O_0 \otimesk \rho$}
\label{subsec:proj to O0rho}Let $A=D(Y)$, $B=D^G_c(U)$, 
$F=\Phi_{O_\cZ}$ and $H=\Phi_{O_\cZ}^*$. 
We apply 
\cite[Proposition 1.5]{BK} to them. 
Let $\image(F)=\{b\in B; b\simeq Fa\ \text{for some $a\in A$}\}$ and 
$\Ker(H)=\{c\in B; Hc\simeq 0\}$. 
By Corollary \ref{cor:Phi and PhiOZ}, $F=\Phi_{O_\cZ}$ is fully faithful because 
$\Phi$ is fully faithful by Theorem \ref{thm:Ishii fully faithful}. 
Moreover $F=\Phi_{O_\cZ}:A\to B$ has a right adjoint 
$H=\Phi_{O_\cZ}^*:B\to A$ by Corollary \ref{cor:rightadj of PhiOZ}. 
Hence $HF\simeq\id_A$. 
It follows that there is a semi-orthogonal decomposition $(\Ker(H),\image F)$ 
of $B=D^G_c(U)$. Therefore $HF$ is the projection of $B$ to the subcategory 
$\image(F)$, 
which we call {\em a semi-orthogonal projection}. \par
Let $\rho$ be an irreducible representation of $G$. 
In what follows we study the semi-orthogonal projection 
$HF=\Phi_{O_\cZ}\circ \Phi_{O_\cZ}^*$ of $O_0 \otimesk \rho$.
In Subsection~\ref{subsec:proj to O0rho} 
we study the image by $HF$ of $O_0 \otimesk \rho$ 
for $\rho$ non-trivial special, while we 
study the image by $HF$ of $O_0 \otimesk \rho_0$ for $\rho_0$ trivial
in Subsection~\ref{subsec:proj to O0rho0}. Note that if $\rho$ is non-special, this vanishes by Corollary \ref{cor:Phi(O0rho)}.\par
In this subsection, we assume that $\rho$ is non-trivial special.  
We shall prove Theorem \ref{thm:socle-special} 
analogous to \cite[Theorem~7.2]{Ishii02}.

We quote \cite[Lemma 2.4]{Ishii-Ueda15}:
\begin{sublemma}\label{lemma:Ueda-Ishii}
Let $f:Y\to X$ be a proper surjective morphism of surfaces 
with $X$ affine.  
Let $\cE$ and $\cF$ be coherent $O_Y$-modules. 
If 
$\cE$ is generated by global sections and if $H^1(Y,\cF)=0$, 
then $H^q(Y,\cE\otimes_{O_Y}\cF)=0$ for $q>0$.
\end{sublemma}

\begin{subcor}
\label{cor:Ueda-Ishii}
Let $f:Y\to X$ be a minimal resolution 
of a normal affine surface $X$ with 
a rational singular point $P$, $F$ the fundamental divisor of $Y$ and 
$\cN$ a coherent sheaf on $Y$. 
Then $H^1(Y, \cN\otimes_{O_Y}O_F)=0$ if and only if $H^1(Y, \cN)=0$. 
\end{subcor}
\begin{proof}
The proof is the same as in \cite[Lemma~3.1]{Ishii92},
where we don't need the assumption that $\cE$ (or $\cN$ here) is locally free.
\end{proof}

\begin{sublemma}\label{lemma:sheaf}
If $H^1(Y,\cC)=0$,
then $\Phi_{\cO_{\cZ}}(\cC)$ is a sheaf such that 
\begin{equation*}\label{eq:piphi}
\pi_*(\Phi_{O_\cZ}(\cC))
\simeq \bigoplus_{\sigma \in \Irr(G)} f_*(\cC \otimes_{O_Y} 
\cM_{\sigma}) \otimesk  \sigma.
\end{equation*}
\end{sublemma}
\begin{proof}
By pushing forward to $X$, we obtain
\begin{equation}\label{eq:pi_*}
\begin{aligned}
\bR\pi_*(\Phi_{O_\cZ}(\cC))
&=\bR(\pi\circ\pi_U)_*(\pi_Y^*\cC 
\otimes_{O_{Y\times_kU}} O_\cZ) \\
&\simeq \bR f_*(\cC \otimes_{O_Y} p_*O_\cZ)  \\
&\simeq \bigoplus_{\sigma \in \Irr(G)} \bR f_*(\cC \otimes_{O_Y} 
\cM_\sigma) \otimesk  \sigma
\end{aligned}
\end{equation}from Lemma~\ref{sublemma:direct image of OZ and OU}.
Since $\cM_\sigma$ is a full $O_Y$-module, which is generated by global sections
by definition,
we have $H^q(\cC \otimes_{O_Y} \cM_\sigma)=0$ for $q > 0$
by Lemma~\ref{lemma:Ueda-Ishii}.
Hence $\bR\pi_*(\Phi_{O_\cZ}(\cC))=\pi_*(\Phi_{O_\cZ}(\cC))$ is a sheaf.
This implies that $\Phi_{O_\cZ}(\cC)$ is a sheaf
because $\pi$ is an affine morphism.
\end{proof}

\begin{sublemma}
\label{lemma:Phi(Mrho dual)}
$\Phi_{O_\cZ}(\cM_\rho^\vee) \simeq O_U \otimesk  \rho$  for $\rho\in\Irr(G)$ special.
\end{sublemma}
\begin{proof}Since $\rho$ is special, 
$\Phi_{O_\cZ}(\cM_\rho^\vee)$ is a sheaf by Lemma~\ref{lemma:sheaf}.
Since $\cM_{\sigma}$ is full for any $\sigma\in\Irr(G)$  
by Lemma~\ref{sublemma:direct image of OZ and OU}~(4), 
we have $H^1(\cM_{\sigma}^{\vee}\otimes_{O_Y}\omega_Y)=0$. 
Hence by Lemma~\ref{lemma:Ueda-Ishii}, 
$H^1(\cM_{\rho}\otimes_{O_Y}\cM_{\sigma}^{\vee}\otimes_{O_Y}\omega_Y)=0$.
Therefore $\cM_{\rho}^{\vee}\otimes_{O_Y}\cM_{\sigma}$ satisfies the conditions {(i) and (iii)} of Lemma~\ref{sublemma:basics}~(1). Hence 
$f_*(\cM_{\rho}^{\vee}\otimes_{O_Y}\cM_{\sigma})$ is reflexive 
by Lemma~\ref{sublemma:basics}~(2).
It follows that
$\pi_*(\Phi_{O_\cZ}(\cM_\rho^\vee))$
is a reflexive $O_X$-module, which is characterized by the vanishing 
$$0=H^1_{\{0\}}(X, \pi_*(\Phi_{O_\cZ}(\cM_{\rho}^{\vee}))) 
\simeq H^1_{\{0\}}(U, \Phi_{O_\cZ}(\cM_\rho^\vee)),$$
and therefore  
$\Phi_{O_\cZ}(\cM_\rho^\vee)$ is a reflexive $O_U$-module.\par
Since $U$ is regular, 
$\Phi_{O_\cZ}(\cM_\rho^\vee)$ is a locally free $O_U$-module 
by Lemma~\ref{sublemma:reflexive} (2).
It follows 
from \eqref{eq:pi_*}, Lemma~\ref{sublemma:direct image of OZ and OU} 
and  Corollary~\ref{cor:dual}
that
$$\pi_*(\Phi_{O_\cZ}(\cM_\rho^\vee))^G \simeq \bigoplus_{\sigma \in \Irr(G)} (f_*(\cM_\rho^\vee \otimes_{O_Y} \cM_\sigma)\otimesk  \sigma)^G=f_*(\cM_\rho^\vee)\simeq M_{\rho^*}.$$

By Lemma \ref{lemma:equivalence},  
$\Phi_{O_\cZ}(\cM_\rho^\vee)$ is uniquely determined by 
the reflexive module $\pi_*(\Phi_{O_\cZ}(\cM_\rho^\vee))^G$. 
Since $(O_U\otimesk \rho)^G=M_{\rho^*}$, we obtain the assertion.
\end{proof}

\begin{subdefn}
For special representations $\rho$ and $\sigma$, we define a non-negative integer $a_{\rho\sigma}$ as follows:
$$
a_{\rho\sigma}=\begin{cases}
E(\rho)E(\sigma) & \text{$\rho$, $\sigma$ non-trivial, $\rho\ne \sigma$}\\
-FE(\rho) &\rho \ne \rho_0,\, \sigma=\rho_0 \\
\max{\{-(K_Y+F)E(\rho), 0\}} & \rho=\rho_0,\, \sigma\ne \rho_0 \\
0 & \rho=\sigma
\end{cases}
$$
\end{subdefn}

\begin{subprop}\label{prop:Phi(OErho(-1))}
Let $\rho\in\Irr(G)$ be non-trivial special, and 
$E(\rho)$ the corresponding exceptional curve.
Then, the following hold:
\begin{enumerate}
\item $\Phi_{O_\cZ}(O_{E(\rho)}(-1))$ is 
a quotient sheaf of $O_U \otimesk \rho$,
\item let $J(\rho)$ be the kernel of the surjection 
$O_U \otimesk \rho \to \Phi_{O_\cZ}(O_{E(\rho)}(-1))$.
Then 
\begin{align*}
J(\rho)/\fm J(\rho)&\simeq 
\bigoplus_{
\begin{matrix}
\sigma\in\Irr(G):\text{special}
\end{matrix}}
\sigma^{\oplus a_{\rho,\sigma}},\\
O_U \otimesk \rho/J(\rho)&\simeq \Phi_{O_\cZ}(O_{E(\rho)}(-1)),
\end{align*}
\item 
$J(\rho)$ is the $O_U$-submodule of $O_U \otimesk  \rho$ generated 
by all the special representations in $\fm \otimesk  \rho$.
\end{enumerate}
\end{subprop}
\begin{proof}By Lemma~\ref{sublemma:basics} and 
Theorem \ref{thm:special mckay corresp}, there is an isomorphism 
\begin{equation}\label{eq:split}
\cM_\rho|_{E(\rho)} \simeq O_{E(\rho)}(1) 
\oplus O_{E(\rho)}^{\oplus \dim \rho-1}.
\end{equation}
Hence there is a surjection
$j(\rho):\cM_\rho^\vee\to O_{E(\rho)}(-1)$. Let 
$\cK_\rho=\Ker(j(\rho))$. Then there is an exact sequence
\begin{equation}\label{eq:Krho}
0 \to \cK_\rho \to \cM_\rho^\vee \to O_{E(\rho)}(-1) \to 0.
\end{equation}
Since $\rho$ is special, $H^1(\cM_{\rho}^{\vee})=0$. 
Hence $H^1(Y,\cK_\rho)=0$ by \eqref{eq:Krho}. 
Thus Lemma~\ref{lemma:sheaf} implies that
any of $\Phi_{O_\cZ}(\cK_\rho)$, $\Phi_{O_\cZ}(\cM^{\vee})$ 
and $\Phi_{O_\cZ}(O_{E(\rho)}(-1))$
is a sheaf at degree 0. 
Applying $\Phi_{O_\cZ}$ to \eqref{eq:Krho},  
we obtain a distinguished triangle
\begin{equation*}
\Phi_{O_\cZ}(\cK_\rho) \to \Phi_{O_\cZ}(\cM_\rho^\vee) \to \Phi_{O_\cZ}(O_{E(\rho)}(-1))
\to \Phi_{O_\cZ}(\cK_\rho)[1],
\end{equation*}which reduces to
 a short exact sequence of sheaves
$$
0 \to \Phi_{O_\cZ}(\cK_\rho) \to \Phi_{O_\cZ}(\cM_\rho^\vee) \to \Phi_{O_\cZ}(O_{E(\rho)}(-1))
\to 0.
$$

Thus $\Phi_{O_\cZ}(O_{E(\rho)}(-1))$ is a quotient sheaf of
$\Phi_{O_\cZ}(\cM_\rho^\vee)\simeq O_U \otimesk  \rho$ 
by Lemma~\ref{lemma:Phi(Mrho dual)}. 
Hence (1) is proved.

Let $\sigma$ be an irreducible representation of $G$.
Then  
the multiplicity of $\sigma$ in $J(\rho)/\fm J(\rho)$ 
is given by the dimension of
\begin{align*}
\Ext^1(O_U \otimesk  \rho/J(\rho), O_0\otimesk  \sigma)
&= \Ext^1(\Phi_{O_\cZ}(O_{E(\rho)}(-1)), O_0\otimesk  \sigma) \\
&\simeq \Ext^1(O_{E(\rho)}(-1), \Phi_{O_\cZ}^*(O_0 \otimesk  \sigma)).
\end{align*}
If $\sigma$ is non-special, then this vanishes by Corollary \ref{cor:Phi(O0rho)}.
If $\sigma$ is special non-trivial, 
$\dim \Ext^1(O_{E(\rho)}(-1), \Phi_{O_\cZ}^*(O_0 \otimesk  \sigma))
=a_{\rho,\sigma}$  follows from Corollary \ref{cor:Phi(O0rho)} and 
Lemma~\ref{lemma:cup product coh}. If $\sigma=\rho_0$, then 
by Corollary \ref{cor:Phi(O0rho)}
\begin{align*}
\Ext^1(O_{E(\rho)}(-1), \Phi_{O_\cZ}^*(O_0 \otimesk  \sigma))&
\simeq \Ext^1(O_{E(\rho)}(-1), \omega_F[1])\\
&\simeq \Ext^2(O_{E(\rho)}(-1), (\omega_Y+F)\otimes_{O_Y} O_F)\\
&\simeq \Ext^0(O_F(F), O_{E(\rho)}(-1))^{\vee}\\
&\simeq \Ext^0(O_F, O_{E(\rho)}(-1-FE(\rho)))^{\vee},
\end{align*}whose dimension is equal to $a_{\rho,\rho_0}$.  Hence we derive (2). 

Finally we shall prove (3). 
If $\sigma=\rho$, then
$H^0(\cM_\rho \otimes_{O_Y} O_{E(\rho)}(-1)) \simeq k$  by \eqref{eq:split}. 
Meanwhile
for any special $\sigma\ (\ne \rho)$, we have
$\cM_{\sigma}|_{E(\rho)} \simeq O_{E(\rho)}^{\oplus \dim \sigma}$ and 
$H^0(\cM_\rho \otimes_{O_Y} O_{E(\rho)}(-1))=0$. 
By Lemma~\ref{lemma:sheaf} 
\begin{equation}\label{eq:pi*Phi(OErho(-1))2}
\pi_*\Phi_{O_\cZ}(O_{E(\rho)}(-1)) 
\simeq \rho\oplus 
\bigoplus_{
\begin{matrix}
\sigma \in \Irr(G):
\text{non-special}
\end{matrix}} H^0(\cM_\sigma \otimes_{O_Y} O_{E(\rho)}(-1)) \otimesk  \sigma.
\end{equation}

Therefore $(\fm \otimesk  \rho)/J(\rho)$ 
contains no special representations. 
Since $J(\rho)$ is generated by special representations by (2), we obtain (3).
\end{proof}

By \cite[Theorem 7.2]{Ishii02},  
$y\in E(\rho)$ if and only if 
$\rho\otimes\rho_{\det}\subset\Soc(O_{Z_y})$. 
{Corollary~\ref{cor:nonspecial repres} shows that,  
if $G\not\subset\SL(2,k)$, and 
if $\rho$ is non-trivial and special, then 
 $\rho\otimes\rho_{\det}$ 
is non-special.} In contrast, 
Theorem~\ref{thm:socle-special}
 characterizes $y\in E(\rho)$ 
by the existence of a certain 
$k[G]$-submodule $(\simeq \rho)$ of $O_{Z_y}$, which is regarded 
as a natural generalization of a socle of $O_{Z_y}$. 

\begin{subthm}\label{thm:socle-special}
Let $\rho\in\Irr(G)$ be non-trivial special.
For $y \in F$, the following conditions are equivalent:
\begin{itemize}
\item[(1)]
$y \in E(\rho)$.
\item[(2)]
There is a $k[G]$-submodule $V(\rho) \subset O_{Z_y}$ isomorphic to $\rho$
such that the $O_U$-submodule of $O_{Z_y}$ generated by $V(\rho)$ contains no other
special representations. 
\end{itemize}
Moreover, $V(\rho)$ is unique if it exists.
\end{subthm}
\begin{proof}
Since $\Phi_{O_\cZ}$ is fully faithful, we have by $\Phi_{O_\cZ}(O_y)=O_{Z_y}$
\begin{equation}\label{eq:one-dim}
\begin{aligned}
\Hom_{D^G(U)}&(O_U\otimesk \rho/J(\rho),O_{Z_y})\\
&\simeq 
\Hom_{D^G(U)}(\Phi_{O_\cZ}(O_{E(\rho)}(-1)), \Phi_{O_\cZ}(O_y))\\
&\simeq \Hom_{D(Y)}(O_{E(\rho)}(-1), O_y).
\end{aligned}
\end{equation}

Assume $y \in E(\rho)$. By (\ref{eq:one-dim})
the inclusion 
$\{y\} \hookrightarrow E(\rho)$ induces a non-zero $G$-homomorphism
\begin{equation}\label{eq:nonzero homom zeta_y}
\zeta_y: O_U\otimesk \rho/J(\rho)\to O_{Z_y}. 
\end{equation} Let $V(\rho)$ be 
the image of the $k[G]$-submodule $k\otimesk \rho$ 
of $O_U\otimesk \rho/J(\rho)$ by {$\zeta_y$}. Since $k\otimesk \rho$ generates  
$O_U\otimesk \rho/J(\rho)$, {it is clear that} $V(\rho)\simeq\rho$.  
Then (2) follows from Proposition~\ref{prop:Phi(OErho(-1))}~(3). 
Since \eqref{eq:one-dim} is one-dimensional, $V(\rho)$ is unique. \par
Conversely suppose that (2) is true. Any $k[G]$-isomorphism $k\otimesk  \rho\simeq V(\rho)$ extends to an $O_U$-homomorphism from 
$O_U\otimesk \rho$ to $O_{Z_y}$ uniquely, which 
  annihilates $J(\rho)$ by 
Proposition~\ref{prop:Phi(OErho(-1))}~(3) and therefore induces an $O_U$-homomorphism 
from $O_U\otimesk \rho/J(\rho)$ to $O_{Z_y}$. Hence it induces a 
non-zero $O_Y$-homomorphism from $O_{E(\rho)}(-1)$ 
to $O_y$ by (\ref{eq:one-dim}), so 
$y\in E(\rho)$. This proves (1). 
\end{proof}

\begin{subdefn}\label{def:mono-special}
 We call 
the submodule $O_U V(\rho) \subset O_{Z_y}$ {\em the mono-special $O_U$-submodule of $O_{Z_y}$} 
(associated to $\rho$). 
\end{subdefn}
Note that no socle of $O_{Z_y}$ is a mono-special $O_U$-submodule of $O_{Z_y}$ 
if $G\not\subset\SL(2)$. If $G\subset\SL(2)$, $V(\rho)$ is a socle of 
$O_{Z_y}$ if and only if it is a mono-special $O_U$-submodule of $O_{Z_y}$.

Assume $y \in E(\rho)$. Then by (\ref{eq:nonzero homom zeta_y})
$\zeta_y$ induces 
an $O_U$-homomorphism 
\begin{equation}\label{eq:alpha_y}
\alpha_y: O_U \otimesk \rho \to O_{Z_y}
\end{equation}
sending $k \otimesk \rho \subset O_U \otimesk \rho$
to $V(\rho)$. 
Then $J(\rho)$ can be described as follows.
\begin{sublemma}
$\displaystyle{J(\rho)=\bigcap_{y \in E(\rho)} \ker \alpha_y}$
\/for $\rho\in\Irr(G)$ non-trivial special.
\end{sublemma}
\begin{proof}
The inclusion $J(\rho) \subset \ker \alpha_y$ 
is obvious.
Decompose $\alpha_y$ as
$$
O_U \otimesk \rho \to O_U \otimesk \rho/J(\rho) \simeq \Phi_{O_{\cZ}}(O_{E(\rho)(-1)}) \overset{\eta_y}{\to} \Phi_{O_{\cZ}}(O_y)\simeq O_{Z_y}.
$$where $\eta_y$ is the restriction map induced 
from the inclusion 
$\{y\} \hookrightarrow E(\rho)$.

 We recall by Lemma~\ref{lemma:sheaf} 
\begin{equation}\label{eqs:PhiErho(-1) Phi(Oy)} 
\begin{aligned}
\pi_*\Phi_{O_{\cZ}}(O_{E(\rho)}(-1)) 
&= \bigoplus_{\sigma \in \Irr(G)} 
H^0(\cM_\sigma \otimes_{O_Y} O_{E(\rho)}(-1)) \otimesk  \sigma,\\
\pi_*\Phi_{O_{\cZ}}(O_y)
&= \bigoplus_{\sigma \in \Irr(G)} 
H^0(\cM_\sigma \otimes_{O_Y} O_y) \otimesk  \sigma.
\end{aligned}
\end{equation}
Let $W=\bigcap_{y\in E(\rho)} \ker \eta_y$. 
To prove the assertion, it suffices to check 
$W=0$. Let $\phi=\sum_{\sigma\in\Irr(G)}\phi_{\sigma}\otimes\sigma\in W$ 
for $\phi_{\sigma}\in 
H^0(\cM_\sigma \otimes_{O_Y} O_{E(\rho)}(-1))$. 
Then the evaluation of $\phi_{\sigma}$ at every $y\in E(\rho)$  is 
equal to zero, hence 
$\phi_{\sigma}$ is zero because 
$\cM_\sigma \otimes_{O_Y} O_{E(\rho)}(-1)$ is 
locally $O_{E(\rho)}$-free.  Hence $\phi=0$ so that 
$W=0$. This completes the proof.
\end{proof}

\subsection{Semiorthogonal projection of $O_0 \otimesk  \rho_0$.}
\label{subsec:proj to O0rho0}
In this subsection, we consider $\Phi_{O_\cZ}(\omega_F[1])=\Phi_{O_\cZ}\circ \Phi_{O_\cZ}^*(O_0 \otimesk  \rho_0)$.
\begin{sublemma}\label{lemma:Phi^i(omega)}
Let $\Phi^i_{O_\cZ}(\omega_F[1])$ be the $i$-th cohomology sheaf 
of $\Phi_{O_\cZ}(\omega_F[1])$.  Then $\Phi^i_{O_\cZ}(\omega_F[1])$ is a $k[G]$-module such that:
\begin{enumerate}
\item
$\Phi_{O_\cZ}^i(\omega_F[1])=0$ if $i\neq 0,-1$,
\item
$\Phi_{O_\cZ}^0(\omega_F[1])$ contains $\rho_0$
with multiplicity one and it contains no other special representations,
\item
$\Phi_{O_\cZ}^{(-1)}(\omega_F[1])$ contains no special representations, where $\Phi_{O_\cZ}^{(-1)}(\omega_F[1])=0$ if and only if $G\subset \SL(2,k)$. 
\end{enumerate}
\end{sublemma}
\begin{proof} 
By  \eqref{eq:pi_*}, 
we have $\Phi_{O_\cZ}^i(\omega_F[1])
=\Phi_{O_\cZ}^{i+1}(\omega_F)=0$ for 
$i\neq -1,0$, hence (1) is proved. 
\par
By the Serre duality for $F$, we have
\begin{equation}\label{eq:pi_Phiomega_F[1]}
\begin{aligned}
\pi_*\Phi_{O_\cZ}^{i}(\omega_F[1])
&=\bigoplus_{\rho\in\Irr(G)}
H^{i+1}(\cM_\rho \otimes_{O_Y} \omega_F) \otimesk  
\rho\\
&\simeq \bigoplus_{\rho\in\Irr(G)}
H^{-i}(\cM_\rho^\vee\otimes_{O_Y}O_F)^\vee\otimesk  
\rho.
\end{aligned}
\end{equation}

Now we shall prove (2). 
The multiplicity of the trivial representation $\rho_0$ 
in $\Phi_{O_\cZ}^0(\omega_F[1])$ is equal to 
$\dim H^0(\cM_{\rho_0}^\vee\otimes_{O_Y}O_F) = \dim H^0(O_F)=1$ 
by \cite[Theorem 4, p.~132]{Artin66}.
We note that $H^1(O_F)=0$ so that $\chi(O_F)=1$. 
Let $\rho$ be a non-trivial special representation.
Then $\rank(\cM_\rho)=c_1(\cM_\rho)F$ by 
Theorem \ref{thm:special mckay corresp}~(2) (see \cite{Wunram88}).
Then by Riemann-Roch, $$
\chi(\cM_\rho^\vee\otimes_{O_Y}O_F)=
-c_1(\cM_\rho)F+(\rank(\cM_\rho))\chi(O_F)=0.
$$ 
By $H^1(\cM_\rho^\vee\otimes_{O_Y}O_F)=0$, 
we have $H^0(\cM_\rho^\vee\otimes_{O_Y}O_F)=0$. 
This proves (2).  \par
Next we shall prove (3). 
By Corollary~\ref{cor:Ueda-Ishii},  we see the following:
\begin{itemize}
\item[--] 
$H^1(\cM_\rho^\vee\otimes_{O_Y}O_F)=0$ 
if and only if $\rho$ is special, that is, $H^1(\cM_\rho^\vee)=0$,
\item[--] $H^1(\cM_\rho^\vee\otimes_{O_Y}O_F)\neq 0$ if and only if 
$\rho$ is non-special.
\end{itemize}
Hence \eqref{eq:pi_Phiomega_F[1]} implies
\begin{itemize}
\item[--] there is no  $\rho$-part 
of $\pi_*\Phi_{O_\cZ}^{(-1)}(\omega_F[1])$ 
for $\rho$ special,
\item[--]
the $\rho$-part 
of $\pi_*\Phi_{O_\cZ}^{(-1)}(\omega_F[1])$ is non-zero
for $\rho$ non-special.
\end{itemize}
Therefore, $\Phi_{O_\cZ}^{(-1)}(\omega_F[1]) \ne 0$ if and only if
there is a non-special representation of $G$,
which is equivalent to $G \not\subset \SL(2,k)$ by Corollary~\ref{cor:nonspecial repres}.
This proves (3).
\end{proof}
\begin{subprop}\label{prop:spectral}
For $\rho\in\Irr(G)$,
we have 
\begin{align*}
&\dim\Ext^p_{D^G_c(U)}(\Phi_{O_\cZ}^0(\omega_F[1]), O_0 \otimesk \rho) \\
& \qquad \qquad=
{\begin{cases}
1& \text{if $p=0$, $\rho=\rho_0$},\\
\max\{-(K_Y+F)E(\rho), 0\} & 
\text{if $p=1$, $\rho$: non-trivial special,} \\
0 & \text{otherwise}.
\end{cases}}
\end{align*}
\end{subprop}
\begin{proof}For a fixed $\rho\in\Irr(G)$, 
we consider the spectral sequence 
\begin{equation*}
E_2^{p,q}=\Ext^p_{D^G_c(U)}(\Phi^{(-q)}_{O_\cZ}(\omega_F[1]), 
O_0\otimesk \rho) 
\end{equation*}
with abutment
\begin{equation*}
E_\infty^{p+q}=\Hom^{p+q}_{D^G_c(U)}
(\Phi_{O_\cZ}(\omega_F[1]), O_0\otimesk \rho).
\end{equation*}
%
Since the global dimension of $\Coh^G U$ is $2$, we have
$E_2^{p,q} = 0$ unless $0 \le p \le 2$ and $0 \le q \le 1$. 
Then the standard property of the spectral sequence {implies}
$E_2^{0,0}=E_\infty^0$ 
and $E_2^{1,0} \subset E_{\infty}^{1,0}\subset E_{\infty}^1$. \par
Now we compute $E_\infty^i$ $(i=0,1)$. 
By Corollary \ref{cor:Phi(O0rho)},
\begin{align*}
E_\infty^p&\simeq
\Hom^{p}_{D(Y)}(\omega_F[1], \Phi_{O_\cZ}^*(O_0 \otimesk  \rho))\\
&\simeq 
\begin{cases}
\Hom^{p}_{D(Y)}(\omega_F[1], \omega_F[1])&\text{if $\rho=\rho_0$},\\
\Hom^{p-1}_{D(Y)}(\omega_F, O_{E(\rho)}(-1))&
\text{if $\rho$ is non-trivial special,}\\
0&\text{if $\rho$ is nonspecial.}
\end{cases}
\end{align*}

For $\rho=\rho_0$, we have $E_\infty^0\simeq H^0(O_F)\simeq k$ and 
$E_\infty^1\simeq \Ext^1_{O_Y}(O_F, O_F)=0$.  
 If $\rho$ is non-trivial special, then $E_\infty^0=0$ and 
$E_\infty^1\simeq H^1(O_{E(\rho)}(-1)\otimes_{O_F}\omega_F^{-1})$. 
This completes the proof. 
\end{proof}

\begin{subcor}\label{cor:J(rho0)} The following holds:
\begin{itemize}
\item[1.] $\Phi^0_{O_\cZ}(\omega_F[1])$ is 
a quotient sheaf of $O_U \otimesk \rho_0$.
\item[2.] Let $J(\rho_0)$ be the kernel of the surjection 
$O_U \otimesk \rho_0 \to \Phi^0_{O_\cZ}(\omega_F[1])$.
Then \begin{itemize}
\item $J(\rho_0)$ contains 
all special representations in $\fm\otimesk \rho_0$,
\item 
the minimal generators of 
$J(\rho_0)$ are all non-trivial special representations $\sigma$ in $\fm\otimesk  \rho_0$ with $(K_Y+F)E(\sigma) < 0$, 
each with multiplicity $-(K_Y+F)E(\sigma)$.
\end{itemize}
\end{itemize}
\end{subcor}
\begin{proof}Let $\cF=\Phi^0(\omega_F[1])$. Proposition~\ref{prop:spectral} 
shows $\cF/\fm \cF\simeq \rho_0$, 
which implies that $\cF$ is generated over $O_U$ by a single element of $\cF$
since $\cF$ 
is supported by the origin of $U$.
In other words, there is a surjection $f:O_U\otimes\rho_0\to \cF$.  
This proves (1). Let $J(\rho_0)=\ker(f)$. Then there is an exact sequence 
$$0\to J(\rho_0)\to O_U\otimesk \rho_0\to \cF\to 0,$$
from which we infer 
$\Ext^1_{D^G(U)}(\cF,O_0\otimesk \rho_0)\simeq
\Hom_{D^G(U)}(J(\rho_0),O_0\otimesk \rho_0)$ because 
$\Ext^1_{D^G(U)}(O_U,O_0\otimesk \rho_0)=H^1(U, O_0\otimesk \rho_0)^G=0$.  Thus 
(2) follows from Lemma~\ref{lemma:Phi^i(omega)}~(2), 
Proposition~\ref{prop:spectral} and 
$$\Hom_{D^G(U)}(J(\rho_0),O_0\otimesk \rho_0)\simeq 
\Hom_{D^G(U)}(J(\rho_0)/\fm J(\rho_0),O_0\otimesk \rho_0).
$$
\end{proof}

Proposition \ref{prop:Phi(OErho(-1))}~(2) and Corollary \ref{cor:J(rho0)} are summarized as follows:
\begin{subcor}\label{cor:sharp arrows}
For $\rho\in\Irr(G)$ special, there is an $k[G]$-isomorphism
$$
J(\rho)/\fm J(\rho) \simeq \bigoplus_{
\begin{matrix}
\sigma\in\Irr(G):
\text{special}
\end{matrix}} \sigma^{\oplus a_{\rho\sigma}}.
$$
\end{subcor}

\section{Reconstruction algebra and 
deformations of $G$-clusters}
\label{sec:reconst alg and deformations}

\subsection{The reconstruction algebra of $X$}
\label{subsec:reconst quiver}
Let 
$$E=\bigoplus_{\rho\text{:special}} O_U\otimesk \rho.$$ 
An endomorphism of 
$E$ is defined to be a $G$-equivariant 
$O_U$-linear endomorphism of $E$. 
The purpose of this 
subsection is to study  
the endomorphism algebra $\End(E)$ of $E$ and briefly recall 
the reconstruction algebra of $X$ due to Wemyss.

\begin{subrem}
\label{rem:widetilde alpha}
Let $\rho$ and $\sigma\in \Irr(G)$ be both special. 
By Corollary \ref{cor:sharp arrows}
we choose $G$-equivariant $O_U$-homomorphisms, 
$$
\phi^i_{\rho, \sigma} : O_U \otimesk \sigma \to 
\fm\otimesk \rho\subset O_U \otimesk \rho
\quad (1 \le i \le a_{\rho\sigma})
$$
which are lifts of linearly independent $G$-homomorphisms
from $O_U \otimesk \sigma$ to $J(\rho)/\fm J(\rho)$. 
We may assume each $\phi^i_{\rho, \sigma}$ is a homogeneous 
$O_U$-homomorphism (that is, a homomorphism defined by homogeneous polynomials) because the diagonal $\bG_m$ action on $U$ commutes with the action of $G$, hence it lifts to the actions on $X$ and $Y$, therefore every homogeneous term of $\phi^i_{\rho, \sigma}$ is also a homomorphism from  
$O_U \otimesk \sigma$ to $\fm\otimesk \rho\subset O_U \otimesk \rho$. \par 
We define an $O_U$-homomorphism 
$$\widetilde{\alpha}:\bigoplus_{\tau:\text{special}}O_U \otimesk \tau\to J(\rho)$$
by $\widetilde{\alpha}=\sum_{\tau}\phi^i_{\rho, \tau}$. 
Then by Nakayama's lemma,  $\widetilde{\alpha}$ is surjective. 
\end{subrem}

\begin{subprop}\label{prop:endo}
Let $E=E_X:=\bigoplus_{\rho\text{:
\text{special}}} O_U\otimesk \rho$ and $X=U/G$. Then 
$\End(E)$ is an $O_X$-algebra generated by $\phi_{\rho,\sigma}^i$ for $1 \le i \le a_{\rho\sigma}$ and  the identity morphism of 
$O_U\otimesk \rho$ where $\rho$, $\sigma$ range over the set of all special representations of $G$.
\end{subprop}
\begin{proof}
Let $\cA$ be 
the subalgebra of $\bigoplus_{\rho\text{:\text{special}}} O_U\otimesk \rho$
generated by $\{\phi_{\rho,\sigma}^i\}$ 
together with the identity morphisms of $O_U\otimesk \rho$.
Let $f:O_U \otimesk \sigma \to O_U \otimesk \rho$ be a non-zero homomorphism.
We prove by the induction on $\deg f$ 
that if $f$ is homogeneous then $f \in \cA$.
If $\deg f=0$, then $\sigma=\rho$ and $f$ is a scalar multiple of the identity of $O_U\otimesk \rho$.
If $\deg f>0$, then $f$ factors through $J(\rho)$. 
By Remark~\ref{rem:widetilde alpha}, $\widetilde{\alpha}$ is surjective.
Since $O_U\otimesk\sigma$ is projective as an $O_U$-module with $G$-action
\footnote{because $|G|$ is prime to the characteristic of $k$}, 
there exists a $G$-equivariant $O_U$-homomorphism $\psi:O_U\otimesk\sigma\to 
\oplus_{\tau:\text{special}}O_U \otimesk \tau$ such that 
$f=\widetilde{\alpha}\circ \psi$. Let $\psi=\oplus_{\tau:\text{special}}f_{\tau}^i$, where 
$f_{\tau}^i: O_U \otimesk \sigma \to O_U \otimesk \tau$ is a 
$G$-equivariant $O_U$-homomorphism, 
homogeneous of degree
$\deg f - \deg \phi_{\rho, \tau}^i<\deg f$.
Then $f$ is expressed as
$$
f=\sum_{i;\tau:\text{special}} \phi_{\rho, \tau}^i f_{\tau}^i.
$$
By the induction hypothesis, $f_\tau^i\in\cA$ and hence $f\in\cA$. 
\end{proof}

\begin{subdefn}
Following Craw \cite{Craw12} and Wemyss \cite{Wemyss11},
we consider the following quiver $Q=Q_X$:
\begin{itemize}
\item[--] the vertices of $Q$ are the special representations and
\item[--] the number of arrows from $\rho$ to $\sigma$ is $a_{\rho\sigma}$.
\end{itemize}

Following Craw \cite{Craw12}
we call it
the special McKay quiver
of $X$.
\end{subdefn}

\begin{subdefn}
Let $kQ$ be the path algebra of $Q=Q_X$.
Proposition \ref{prop:endo} shows that there
is an isomorphism
$$\End\left(\bigoplus_{\rho\text{:\text{special}}} O_U\otimesk \rho\right) \simeq kQ/\cI,$$
where $kQ$ is an $O_X$-algebra, and 
$\cI$ is a certain two-sided ideal of $kQ$.
The $O_X$-algebra $\End(E_X)$ is called the reconstruction algebra of $X$. 
\end{subdefn}

\begin{subrem}
The generators of the ideal $\cI$ (in the completion $\widehat{kQ}$) 
are given explicitly in the cases of type A and type D by Wemyss \cite{WemyssA, WemyssD1, WemyssD2}.
\end{subrem}
\subsection{$G$-clusters}
In this subsection, we generalize Theorem \ref{thm:mckay tensor prod}.\par
As before we use the following notation:
\begin{align*}
\Ext^k(A,B):&=\Hom^k_{D(Y)}(A,B),\quad
\Ext^k(C,D):=\Hom^k_{D^G(U)}(C,D),
\end{align*}for $A,B\in D(Y)$ and $C,D\in D^G(U)$.

For $y \in E(\rho)$, recall that there are subrepresentations $W(\rho) \subset \Gen(I_y)$
and $W(\rho_0) \subset \Gen(I_y)$ isomorphic to $\rho$ and $\rho_0$ respectively.
Moreover, there is a mono-special $O_U$-submodule 
$O_UV(\rho)$ of $O_{Z_y}$ 
by Theorem \ref{thm:socle-special}.
\begin{subdefn}
Let $\SSoc(O_{Z_y})$ be the direct sum
of $k[G]$-submodules $V(\rho) \subset O_{Z_y}$
generating mono-special $O_U$-submodules
with $y \in E(\rho)$. 
\end{subdefn}
Take any $k[G]$-submodule $\wV(\rho) \subset O_U$ 
which is a lift of $V(\rho)$.

\begin{subthm}\label{thm:cluster-quiver}
Let $\rho$ and $\sigma$ be special representations, and $\rho\neq\sigma$.
\begin{enumerate}
\item
Suppose $\rho$ and $\sigma$ are non-trivial.
Then $a_{\rho \sigma}$ is equal to the number of
points $y \in F$ such that
\begin{itemize}
\item[--] $W(\sigma) \subset \Gen(I_y)$,
\item[--] $V(\rho) \subset \SSoc(O_{Z_y})$ and
\item[--] $W(\sigma) \subset (\fm \wV(\rho) + \fm I_y)/\fm I_y$.
\end{itemize}
Actually, $a_{\rho \sigma}$ is either $0$ or $1$ and $a_{\rho \sigma}=1$ if and only if
$E(\rho) \cap E(\sigma)$ consists of a unique point $y$.
\item
Suppose $\sigma=\rho_0$.
Then, $a_{\rho\sigma} \ne 0$ if and only if
there exists $y \in F$ such that
\begin{itemize}
\item[--]
$W(\sigma) \subset \Gen(I_y)$, 
\item[--]
$V(\rho) \subset \SSoc(O_{Z_y})$, and
\item[--]
$W(\sigma) \subset (\fm \wV(\rho) + \fm I_y)/\fm I_y$.
\end{itemize}
If this holds for one $y$, then
it holds for any $y \in E(\rho)$.\footnote{ 
$W(\sigma)$ depends on $y$. See Subsection~\ref{subsec:dihedral case}. }
\item Suppose $\rho=\rho_0$ and $\sigma\neq\rho_0$.
Then $a_{\rho\sigma}\leq 2$, equality holding if and only if 
$F=E(\sigma)$, and $G$ is a cyclic group $\langle \frac{1}{n}(1,1) \rangle$. 
Moreover,
\begin{itemize} 
\item[(3a)] if $\dim\sigma=1$, then 
 $a_{\rho\sigma}$ is described as follows:
$$
a_{\rho\sigma}= \begin{cases}
0 &\text{if $W(\sigma) \subset \fm J(\rho_0)/\fm I_y$ for any $y \in E(\sigma)$}\\
2 &\text{if $W(\sigma) \not\subset \fm J(\rho_0)/\fm I_y$ for any $y \in E(\sigma)$}\\
1 &\text{otherwise}.
\end{cases}
$$
\item[(3b)] 
if $\dim \sigma>1$,
then $a_{\rho\sigma}$ is either $0$ or $1$.
The equality $a_{\rho\sigma}=1$ holds if and only if there is a $G$-invariant ideal $J$ of $O_U$ containing $I_y$ for all $y\in F$
such that
as a $k[G]$-module
$H^0(O_U/J)$ is isomorphic to the direct sum of $\rho_0 \oplus \sigma$ and non-special representations.
\end{itemize}
\end{enumerate}
\end{subthm}
\begin{proof}Let $\rho\in\Irr(G)$ be special and non-trivial.  
If $y\in E(\rho)$,  
then $\alpha_y$ in (\ref{eq:alpha_y}) 
induces a homomorphism 
\begin{equation}\label{eq:extmap}
\Ext^1(O_{Z_y}, O_0\otimesk \sigma) \to \Ext^1((O_U \otimesk \rho)/J(\rho), O_0\otimesk \sigma). 
\end{equation} Suppose $W(\sigma)\subset \Gen(I_y)$. For a non-trivial element of $\Ext^1(O_{Z_y}, O_0\otimesk \sigma)\neq 0$,
there is a non-trivial extension $(\ext_{\sigma})$ of Section \ref{subsec:socle OU/mIy}:
\begin{equation}\label{eq:extclass1}
0\to W(\sigma)\simeq I_y /J_{\sigma} \to O_U/J_{\sigma} \to O_{Z_y}\simeq O_U/I_y\to 0,
\end{equation}for some ideal $J_\sigma$ of $O_U$
such that $\fm I_y\subset J_\sigma \subset I_y$. 
The pullback of (\ref{eq:extclass1}) 
by $O_UV(\rho)\hookrightarrow O_{Z_y}$ is given by
\begin{equation}\label{eq:extclass2}
0\to W(\sigma)\simeq I_y /J_{\sigma}\to O_U\wV(\rho)/J_{\sigma} \to O_UV(\rho)\to 0.
\end{equation}
Then we see, in the same manner as in Subsection~\ref{subsec:cup products}, that 
$W(\sigma)\subset(\fm \wV(\rho)  + \fm I_y)/\fm I_y$ if and only if 
(\ref{eq:extclass2}) does not split if and only if
(\ref{eq:extmap}) is non-zero.
Since $O_{Z_y} \simeq \Phi_{O_\cZ}(O_y)$ and $(O_U \otimesk \rho)/J(\rho) \simeq \Phi_{O_\cZ}(O_{E(\rho)}(-1))$,
\eqref{eq:extmap} is isomorphic to
\begin{equation}\label{eq:extmap2}
\Ext^1(O_y, \Phi_{O_\cZ}^*(O_0\otimesk \sigma)) \to \Ext^1(O_{E(\rho)}(-1), \Phi_{O_\cZ}^*(O_0\otimesk \sigma))
\end{equation}
induced by $O_{E(\rho)}(-1) \to O_y$.
If $\sigma \ne \rho_0$, then by Corollary \ref{cor:Phi(O0rho)}, \eqref{eq:extmap2} is
$$
\Ext^1(O_y, O_{E(\sigma)}(-1)) \to \Ext^1(O_{E(\rho)}(-1), O_{E(\sigma)}(-1)),
$$
which is non-zero if and only if $y \in E(\rho) \cap E(\sigma)$.
Thus (1) is proved.

If $\sigma=\rho_0$,  then \eqref{eq:extmap2} is  isomorphic to
\begin{equation}\label{eq:extmap3}
\Ext^2(O_y, \omega_F) \to \Ext^2(O_{E(\rho)}(-1), \omega_F).
\end{equation}
Since $\omega_F \simeq O_F(F) \otimes_{O_Y} K_Y$, the Serre dual of \eqref{eq:extmap3} is
$$
\Hom(O_F(F), O_{E(\rho)}(-1)) \to \Hom(O_F(F), O_y),
$$
which is non-zero if and only if $FE(\rho)<0$.
This proves (2).

Next we shall prove (3). Suppose $\rho=\rho_0$ and $y\in F$. 
Since $I_y$ is generated by special representations 
by Theorem \ref{thm:local McKay},
$I_y$ is contained in $J(\rho_0)$ and $\Hom(O_{Z_y}, O_U/J(\rho_0))$ is a one-dimensional vector space generated by the natural surjection, which is denoted by $\psi$. \par
Now suppose $y\in E(\sigma)$. Consider the cup product
\begin{equation}\label{eq:extmap4}
\begin{aligned}\Hom(O_{Z_y}, O_U/J(\rho_0)) \otimesk  \Ext^1(O_U/J(\rho_0), O_0 \otimesk \sigma)\\
\to \Ext^1(O_{Z_y}, O_0\otimesk \sigma). 
\end{aligned}
\end{equation}

We show that $W(\sigma) \not\subset \fm J(\rho_0)/\fm I_y$ holds if and only if \eqref{eq:extmap4} is a non-zero map.
Notice that $W(\sigma) \not\subset \fm J(\rho_0)/\fm I_y$ holds if and only if the composite
$$
W(\sigma) \hookrightarrow I_y/\fm I_y \to J(\rho_0)/\fm J(\rho_0)
$$
is injective.
It is equivalent to the existence of an ideal $J'$ with $\fm J(\rho_0) \subset J' \subset J(\rho_0)$ such that
$W(\sigma) \to J(\rho_0)/J'$ is an isomorphism,
and therefore to the existence of a commutative diagram of the following form:
\begin{diagram}
0 & \rTo & W(\sigma) & \rTo & O_U/J' & \rTo & O_U/J(\rho_0) & \rTo & 0 \\
&& \Vert && \uTo && \uTo^\psi && \\
0 & \rTo & W(\sigma) & \rTo & O_U/J_\sigma & \rTo & O_{Z_y} & \rTo & 0
\end{diagram}
where the second row corresponds to the generator of $\Ext^1(O_{Z_y}, O_0\otimesk \sigma)$.
Hence the condition $W(\sigma) \not\subset \fm J(\rho_0)/\fm I_y$ is equivalent to the non-vanishing of \eqref{eq:extmap4}.

Now we claim that \eqref{eq:extmap4} is isomorphic to
\begin{equation}\label{eq:extmap5}
\begin{aligned}
\Ext^1(O_y, \omega_F) \otimesk \Hom(\omega_F, O_{E(\sigma)}(-1))\\
\to \Ext^1(O_y, O_{E(\sigma)}(-1)).
\end{aligned}
\end{equation}
By Lemma \ref{lemma:Phi^i(omega)}, $\Phi_{O_\cZ}^{(-1)}(\omega_F[1])$ consists of non-special representations. Hence by Corollary \ref{cor:Phi(O0rho)} $\Phi_{O_\cZ}^*(\Phi_{O_\cZ}^{(-1)}(\omega_F[1]))=0$.
This implies
$$
\Phi_{O_\cZ}^*(O_U/J(\rho_0)) 
\simeq \Phi_{O_\cZ}^*(\Phi_{O_\cZ}^0(\omega_F[1]))\simeq \Phi_{O_\cZ}^*(\Phi_{O_\cZ}(\omega_F[1])) \simeq \omega_F[1]
$$
and therefore
$$
\Hom(O_{Z_y}, O_U/J(\rho_0))\simeq \Hom(O_y, \Phi_{O_\cZ}^*(O_U/J(\rho_0))
\simeq \Ext^1(O_y, \omega_F).
$$
Next by the proof of Proposition~\ref{prop:spectral},
the following is also true:
$$
\Ext^1(O_U/J(\rho_0), O_0 \otimesk \sigma) \simeq 
\Hom(\omega_F, O_{E(\sigma)}(-1)).
$$
We also have by Corollary~\ref{cor:Phi(O0rho)}
$$
\Ext^1(O_{Z_y}, O_0\otimesk \sigma) \simeq \Ext^1(\Phi_{O_\cZ}(O_y), O_0\otimesk \sigma)
\simeq \Ext^1(O_y, O_{E(\sigma)}(-1)).
$$
Thus \eqref{eq:extmap4} is isomorphic to \eqref{eq:extmap5}.

In \eqref{eq:extmap5}, we see
$$
\dim \Hom(\omega_F, O_{E(\sigma)}(-1)) = a_{\rho_0 \sigma}
$$
and hence if \eqref{eq:extmap5} is non-zero, then $a_{\rho_0 \sigma} \geq 1$ 
and $y\in E(\sigma)$.
Especially, if $a_{\rho_0 \sigma} =0$, then we have $W(\sigma) \subset \fm J(\rho_0)/\fm I_y$ for any $y \in E(\sigma)$.

Now let us prove $a_{\rho_0 \sigma} \le 2$.
Since $F$ is the fundamental divisor, we have $(F-E(\sigma))E(\sigma)\geq 0$, with equality holding if and only if $F=E(\sigma)$. Hence
\begin{align*}
a_{\rho_0\sigma}=-(K_Y + F)E(\sigma)&=-\left((K_Y+E(\sigma))E(\sigma)+(F-E(\sigma))E(\sigma)\right) \\
&=-(-2 + (F-E(\sigma))E(\sigma)) \leq 2.
\end{align*}
This implies that $a_{\rho_0 \sigma} \le 2$, and equality holds
if and only if $F=E(\sigma)$. Moreover 
if $F=E(\sigma)$, then it is easy to see that $G$ is a cyclic group $\langle \frac{1}{n}(1,1) \rangle$ where $E(\sigma)^2=-n$, 
$\dim\sigma=1$ for any $\sigma\in\Irr(X)$. 
In this case, one can prove that (\ref{eq:extmap5}) is non-zero, so that
$W(\sigma) \not\subset \fm J(\rho_0)/\fm I_y$ for every $y \in E(\sigma)$.

It remains to consider the case $a_{\rho_0 \sigma}=1$. Then 
$$\Hom(\omega_F, O_{E(\sigma)}(-1))\simeq H^0(O_{E(\sigma)})\simeq k.$$
Therefore by the exact sequences
\begin{align*}
0 \to K_Y \to &K_Y(F) \to \omega_F \to 0 \\
0 \to K_Y(F-E(\sigma)) \to &K_Y(F) \to O_{E(\sigma)}(-1) \to 0,
\end{align*}
we see that 
\eqref{eq:extmap5} is equivalent to
\begin{equation*}\label{eq:extmap6}
\Ext^2(O_y, K_Y) \to \Ext^2(O_y, K_Y(F-E(\sigma))),
\end{equation*}which is isomorphic to 
\begin{equation}\label{eq:extmap7}
\Ext^0(O_Y, O_y)^{\vee} \to \Ext^0(O_Y(F-E(\sigma)), O_y)^{\vee}
\end{equation}
induced by the inclusion $O_Y\hookrightarrow O_Y(F-E(\sigma))$. \par
If $\dim \sigma =1$, then the coefficient of $E(\sigma)$ in $F$ is $1$
and \eqref{eq:extmap7} 
is non-zero if and only if $y \in E(\sigma) \setminus (F-E(\sigma))$.
Therefore, $W(\sigma) \not\subset \fm J(\rho_0)/\fm I_y$ holds for $y \in E(\sigma) \setminus (F-E(\sigma))$ and $W(\sigma) \subset \fm J(\rho_0)/\fm I_y$ holds for $y \in E(\sigma) \cap (F-E(\sigma))$.
\par
Finally suppose $\dim \sigma >1$.
Then $G$ is not a cyclic group and hence $a_{\rho_0\sigma}\ne 0$ implies $a_{\rho_0\sigma}=1$.
In this case, \eqref{eq:extmap7} is zero since the coefficient 
of $E(\sigma)$ in $F$ is greater than $1$
and therefore \eqref{eq:extmap4} is also zero.
Let $J \subset J(\rho_0)$ be the ideal with
$J(\rho_0)/J=O_0\otimesk \sigma^{\oplus a_{\rho_0\sigma}}=O_0\otimesk \sigma$.
Then the vanishing of \eqref{eq:extmap4} implies that $\psi: O_{Z_y} \to O_U/J(\rho_0)$
factors through $O_U/J$, which means $I_y \subset J$.
\end{proof}

\subsection{The integer $a_{\rho \rho_0}$}
Theorem~\ref{thm:cluster-quiver}~(2) 
includes no information about $a_{\rho \rho_0}$, 
the number of arrows
from $\rho$ to $\rho_0$ of the special McKay quiver.
In order to explicitly describe it in terms of $G$-clusters,
we consider Theorem~\ref{thm:cluster-quiver}~(2) for all $y \in E(\rho)$ 
simultaneously.\par
For a scheme $S$ of finite type over $k$,
consider the integral functor
\begin{equation}\label{eq:Phi_S}
\Phi_{O_S \otimesk O_\cZ} :D^b(S \times Y) \to D^b_G(S \times U)
\end{equation}
with kernel object $O_S \otimesk O_{\cZ}=\pi_{YU}^*(O_\cZ) \in D^b_G(S \times Y \times U)$,
where $\pi_{YU}$ is the projection to $Y\times U$.
To be more precise, $\Phi_{O_S \otimesk O_\cZ}$ is defined as
$$
\Phi_{O_S \otimesk O_\cZ}(-)=\bR(\pi_{SU})_*(\bL\pi_{SY}^*(-)\overset{\bL}{\otimes}_{O_{S\times Y \times U}}
\bL\pi_{YU}^*(O_\cZ)).
$$
Its right adjoint $\Phi_{O_Y \otimesk O_\cZ}^*$ is also an integral functor
whose kernel object is $O_S \otimesk (K_Y \otimes_{O_Y} O_\cZ^\vee[2])$, the pull back of
the kernel object of $\Phi_{O_\cZ}^*$.\par
In parallel with $\Phi_{O_\cZ}^* \circ \Phi_{O_\cZ} \simeq I_{D^b(Y)}$, we have
$$
\Phi_{O_Y \otimesk O_\cZ}^* \circ \Phi_{O_Y \otimesk O_\cZ} \simeq I_{D^b(S\times Y)}.
$$
Moreover, the adjointness is refined as follows:
\begin{equation}\label{eq:S-adjoint}
\begin{aligned}
\left(\bR(\pi_S)_*\bR\chom_{O_{S\times U}}(\Phi_{O_Y \otimesk O_{\cZ}}(\alpha), \beta)\right)^G\\
\simeq
\bR(\pi_S)_*\bR\chom_{O_{S \times Y}} (\alpha, \Phi_{O_Y \otimesk {O_\cZ}}^*(\beta)).
\end{aligned}
\end{equation}
For a non-trivial special representation $\rho$,
the $k[G]$-submodule $V(\rho) \subset O_{Z_y}$ depends on $y \in E(\rho)$
and is denoted by $V_y(\rho)$ in the sequel.
\begin{sublemma}
There is an $O_{E(\rho)}$-submodule $\cV(\rho) \subset O_{E(\rho)} \otimes_{O_Y} O_{\cZ}$ (which is not necessarily an $O_{E(\rho)\times U}$-submodule)
such that
$$\cV(\rho) \otimes_{O_{E(\rho)}} O_y = V_y(\rho)$$
for $y \in E(\rho)$
and
$$\cV(\rho) \simeq O_{E(\rho)}(1) \otimes \rho.$$
\end{sublemma}
\begin{proof}
Let $E=E(\rho)$.
We define $\cV(\rho)$ as the image of the evaluation map
\begin{align*}\label{eq:evalmap}
(\pi_{E})_*\chom_{O_{E\times U}}(O_{E} \otimesk (O_U\otimesk \rho /J(\rho)), O_{E\otimes_{O_Y} O_{\cZ}})^G &\otimes_{O_{E}} (O_{E}\otimesk \rho) \\
&\to O_{E}\otimes_{O_Y} O_{\cZ}
\end{align*}
of $O_E$-modules,
where $O_{E}\otimesk \rho \subset O_{E} \otimesk (O_U\otimesk \rho /J(\rho))$ is the tensor product of $O_E$ with the constant term $\rho \subset O_U\otimesk \rho /J(\rho)$.
Taking the $0$-th cohomology sheaves in \eqref{eq:S-adjoint} for $S=E$,
we obtain
$$
\begin{aligned}
(\pi_{E})_* \chom_{O_{E\times U}}&(O_{E} \otimesk (O_U\otimesk \rho /J(\rho), O_{E}\otimes_{O_Y} O_{\cZ})^G \\
&\simeq
(\pi_{E})_* \chom_{O_{E\times Y}}(O_{E}\otimesk O_E(-1),
O_{\Delta_{E}}) \simeq O_{E}(1).
\end{aligned}
$$
Therefore $\cV(\rho)$ is the image of a non-trivial map
\begin{equation}\label{eq:Vrho}
O_{E}(1) \otimesk \rho \to O_E \otimes_{O_Y} O_\cZ,
\end{equation}
which must be injective.
Moreover, by the isomorphisms
$$
O_E \otimes_{O_Y} O_\cZ \cong \bigoplus_{\rho}O_E \otimes_{O_Y}\cM_\rho \otimesk \rho
$$
and $O_E \otimes_{O_Y}\cM_\rho \cong O_E^{\dim\rho-1} \oplus O_E(1)$,
the cokernel of \eqref{eq:Vrho} is flat over $E$.
Therefore, $\cV(\rho) \otimes_{O_{E(\rho)}} O_y \to O_{Z_y}$ is also
injective and this yields $\cV(\rho) \otimes_{O_{E(\rho)}} O_y = V_y(\rho)$.
\end{proof}

By this lemma there is a map
$$
\alpha: O_{E(\rho)}(1)\otimesk O_{U} \otimesk \rho  \to O_{E(\rho)}\otimes_{O_Y} O_{\cZ} 
$$
of sheaves on $E(\rho) \times U$
such that its restriction to the fiber over $y \in E(\rho)$ coincides with $\alpha_y$ in \eqref{eq:alpha_y}.

Recall we have
$$
\cI/(\fm \cI+ I_F) \simeq \left(\sum_{\tau \ne \rho_0} O_{E(\tau)}(-1) \otimesk \tau\right) \oplus O_F(-F) \otimesk \rho_0
$$
by Theorem \ref{thm:mckay isom} and take the ideal $\cI_{\rho, \rho_o}$ with $\fm \cI+ I_F \subset\cI_{\rho, \rho_0} \subset \cI$ such that
$$
\cI/\cI_{\rho, \rho_0} = \cW(\rho_0):=O_{E(\rho)}(-F) \otimesk \rho_0.
$$
Then, by $ H^1(O_{E(\rho)}(-1) \otimes_{O_Y}\cI/\cI_{\rho, \rho_0}) =0$,
$\alpha$ can be lifted to
$$
\tilde{\alpha}:O_{E(\rho)}(1) \otimesk O_{U} \otimesk \rho   \to 
O_{E(\rho)}\otimes_{O_Y}
O_{Y\times U}/\cI_{\rho, \rho_0}.
$$
Since $\alpha$ factors through $O_{E(\rho)}(1) \otimesk (O_{U} \otimesk \rho /J(\rho))$,
$\tilde{\alpha}$ induces a map
\begin{equation}\label{eq:rho-rho0}
O_{E(\rho)}(1) \otimesk (J(\rho)/\fm J(\rho))^G \to (\cI/\cI_{\rho, \rho_0})^G
\end{equation}
which is isomorphic to
$$
O_{E(\rho)}(1)^{\oplus a_{\rho\rho_0}} \to O_{E(\rho)}(a_{\rho\rho_0})
$$
where $a_{\rho\rho_0}=-FE(\rho)$.
By \eqref{eq:rho-rho0} we have $a_{\rho\rho_0}$ maps in
$$
\Hom(O_{E(\rho)}(1), O_{E(\rho)}(a_{\rho\rho_0}))
$$
which corresponds to $a_{\rho\rho_0}$ arrows from $\rho$ to $\rho_0$.
In fact, the following explains the $a_{\rho\rho_0}$ arrows
in terms of $G$-clusters.
\begin{subprop}\label{prop:basis}
The $a_{\rho\rho_0}$ maps given by \eqref{eq:rho-rho0} form a basis of
$$\Hom(O_{E(\rho)}(1), O_{E(\rho)}(a_{\rho\rho_0})).$$
\end{subprop}
\begin{proof}
Put $E=E(\rho)$.
The $O_E$-dual of \eqref{eq:rho-rho0} is
\begin{equation}\label{eq:dualrho-rho0}
\begin{aligned}
(\pi_E)_*\chom_{O_{E \times U}}&(\cI|_{E\times U}, O_{E \times 0} \otimesk \rho_0)^G\\
&\to
(\pi_E)_*\chom_{O_{E \times U}}(O_E(1) \otimesk J(\rho), O_{E \times 0}\otimesk \rho_0)^G,
\end{aligned}
\end{equation}
which is a map of sheaves on $E \simeq E\times 0$.
From the short exact sequence
$$
0 \to O_E(1) \otimesk J(\rho) \to O_E(1) \otimesk (O_U \otimesk \rho) \to  O_E(1) \otimesk(O_U\otimesk \rho/J(\rho)) \to 0
$$
we derive an isomorphism 
\begin{align*}
\chom_{O_{E\times U}}(O_E(1) \otimesk &J(\rho), O_{E \times 0}\otimesk \rho_0) \\
\simeq \cext^1_{O_{E\times U}}&(O_E(1) \otimesk (O_Y\otimesk \rho/J(\rho)), O_{E \times 0} \otimesk \rho_0),
\end{align*}
while from the exact sequence
$$
0 \to \cI|_{E\times U} \to O_{E\times U} \to O_{\cZ}|_{E\times U} \to 0
$$
we derive an isomorphism
\begin{align*}
\chom_{O_{E \times U}}&(\cI|_{E\times U}, O_{E \times 0} \otimesk \rho_0)\\ 
&\simeq
\cext^1_{O_{E \times U}}(O_{\cZ}|_{E \times U}, O_{E \times 0}\otimesk \rho_0).
\end{align*}
Hence \eqref{eq:dualrho-rho0} is isomorphic to
\begin{equation}\label{eq:dual2}
\begin{aligned}
(\pi_{E})_* &\cext^1_{O_{E \times U}}(O_{\cZ}|_{E \times U}, O_{E \times 0}\otimesk \rho_0)^G\\
&\to
(\pi_{E})_*\cext^1_{O_{E\times U}}(O_E(1) \otimesk (O_Y\otimesk \rho/J(\rho)), O_{E \times 0} \otimesk \rho_0)^G
\end{aligned}
\end{equation}
(this is a family of maps \eqref{eq:extmap} parametrized by $y \in E$).
By  \eqref{eq:S-adjoint}, the map \eqref{eq:dual2} is isomorphic to
\begin{equation}\label{eq:extfam2}
\begin{aligned}
(\pi_E)_*&\cext^2_{O_{E \times Y}}(O_{\Delta_E}, O_{E}\otimesk \omega_F)\\
&\to
(\pi_E)_*\cext^2_{O_{E\times Y}}(O_E(1) \otimesk O_E(-1), O_E\otimes \omega_F)
\end{aligned}
\end{equation}
which is a family of maps \eqref{eq:extmap2}.
Taking the dual of \eqref{eq:extfam2}, we see that \eqref{eq:rho-rho0} is isomorphic to
\begin{equation}\label{eq:extfam3}
\begin{aligned}
(\pi_E)_*\chom_{O_{E\times Y}}&(O_E \otimesk \omega_F, O_E(1) \otimesk (O_E(-1) \otimes_{O_Y} K_Y)\\
& \to
(\pi_E)_*\chom_{O_{E\times Y}}(O_E \otimesk \omega_F, O_{\Delta_E}\otimes K_Y)
\end{aligned}
\end{equation}
which is determined by the restriction map
$$
O_E(1) \otimesk O_E(-1) \to (O_E(1) \otimesk O_E(-1))|_{\Delta_E} \simeq O_E.
$$
Since $\omega_F \simeq K_Y(F)|_F$, \eqref{eq:extfam3} is isomorphic to
the evaluation map
$$
\Hom(O_F(-1), O_E(-F))\otimesk O_E(1) \to O_E(-F)
$$
and we are done.
\end{proof}

\section{Examples}
\subsection{The cyclic group 
$G=\left\langle\frac{1}{12}(1,5)\right\rangle$}
Consider the case of a cyclic group
$$G=\left\langle\frac{1}{12}(1,5)\right\rangle=
\left\langle
\begin{pmatrix} \zeta_{12} & 0 \\ 0 & \zeta_{12}^5 
\end{pmatrix}
\right\rangle,$$
where $\zeta_{12}$ is the primitive $12$-th root of unity. 
Let $g=\diag(\zeta_{12},\zeta_{12}^5)$ be a generator of $G$, and 
$\rho_i\in\Irr(G)=\Hom_{\bZ}(G,\bG_m)$ for $i\pmod{12}$
such that $\rho_i(g)=\zeta_{12}^{-i}$. 
Let $S_1$ be the $G$-submodule 
spanned by $x$ and $y$.  The group $G$ acts 
on $U$ diagonally so that 
$S_1\simeq\rho_1\oplus\rho_5$ as $k[G]$-modules, the $\rho_1$-part  
(resp. $\rho_5$-part) being spanned by $x$ (resp. $y$). 
The $k[G]$-module $k[x,y]$ is decomposed into irreducible pieces, each spanned 
by a single monomial. 
Associated to every monomial 
$x^ay^b$ in Figure \ref{fig:monomials weights}, 
we have a representation $\rho_i$ with
$i=a+5b \pmod{12}$ in the same Figure, which is quoted as $i$ instead of $\rho_i$ for short. Figure \ref{fig:monomials weights} 
can be seen as (a covering of) the McKay quiver, where the arrow $i \to i+1 \pmod{12}$ (resp. $i \to i+5 \pmod{12}$) 
corresponds to multiplication by $x$ (resp. $y$).

\begin{figure}\setlength{\unitlength}{2.5pt}
\centering
\hbox{
\hskip 0.7cm\begin{picture}(60,60)(-3,0)
\put(0,0){$1$}
\put(10,0){$x$}
\put(20,0){$x^2$}
\put(30,0){$x^3$}
\put(40,0){$x^4$}
\put(50,0){$\cdots$}
\put(0,10){$y$}
\put(9,10){$xy$}
\put(19,10){$x^2y$}
\put(29,10){$x^3y$}
\put(39,10){$x^4y$}
\put(50,10){$\cdots$}
\put(0,20){$y^2$}
\put(9,20){$xy^2$}
\put(19,20){$x^2y^2$}
\put(29,20){$x^3y^3$}
\put(39,20){$x^4y^3$}
\put(50,20){$\cdots$}
\put(0,30){$y^3$}
\put(9,30){$xy^3$}
\put(19,30){$x^2y^3$}
\put(29,30){$x^3y^3$}
\put(39,30){$x^4y^3$}
\put(50,30){$\cdots$}
\put(0,40){$y^4$}
\put(9,40){$xy^4$}
\put(19,40){$x^2y^4$}
\put(29,40){$x^3y^4$}
\put(39,40){$x^4y^4$}
\put(50,40){$\cdots$}
\put(0,50){$\vdots$}
\put(10,50){$\vdots$}
\put(20,50){$\vdots$}
\put(30,50){$\vdots$}
\put(40,50){$\vdots$}
\end{picture}
\qquad
\begin{picture}(60,60)(0,0)
\put(0,0){0}
\put(10,0){1}
\put(20,0){2}
\put(30,0){3}
\put(40,0){4}
\put(50,0){$\cdots$}
\put(0,10){5}
\put(10,10){6}
\put(20,10){7}
\put(30,10){8}
\put(40,10){9}
\put(50,10){$\cdots$}
\put(-1,20){10}
\put(9,20){11}
\put(20,20){0}
\put(30,20){1}
\put(40,20){2}
\put(50,20){$\cdots$}
\put(0,30){3}
\put(10,30){4}
\put(20,30){5}
\put(30,30){6}
\put(40,30){7}
\put(50,30){$\cdots$}
\put(0,40){8}
\put(10,40){9}
\put(19,40){10}
\put(29,40){11}
\put(40,40){0}
\put(50,40){$\cdots$}
\put(0,50){$\vdots$}
\put(10,50){$\vdots$}
\put(20,50){$\vdots$}
\put(30,50){$\vdots$}
\put(40,50){$\vdots$}
\end{picture}
}
\caption{monomials and weights}\label{fig:monomials weights}
\end{figure}

The special representations 
of $G$ are $\rho_0$, $\rho_1$, $\rho_3$ and $\rho_5$,
which are determined by the continued fraction expansion 
of $\frac{12}{5}$ (\cite{Wunram88}). 
The modules $O_U\otimesk \rho_i/J(\rho_i)$ 
and the generators of $J(\rho_i)$ for the special representations 
are given in Figure~\ref{fig:J0135}.
The framed monomials in those Figures 
form a basis of 
the vector space $O_U\otimesk \rho_i/J(\rho_i)$
and the monomials outside the frame are 
minimal generators of $J(\rho_i)$.

If we look at Figure \ref{fig:J0135}, 
then we find three paths in the covering 
of the McKay quiver from the vertex $\rho_1$, the generator of
$O_U \otimesk \rho_1$, to two $\rho_0$'s 
outside of $O_U\otimesk \rho_i/J(\rho_i)$.
The two paths from $\rho_1$ to $\rho_0$ in the McKay 
quiver corresponds to $\phi^1_{1,0}$ and $\phi^2_{1,0}$ in Subsection 
\ref{subsec:reconst quiver}, which determine the two arrows
from $\rho_1$ to $\rho_0$ in the special McKay quiver $Q_X$
associated to the reconstruction algebra.
From Figure~\ref{fig:J0135} we can read 
the matrix $(a_{\rho_i\rho_j})$ as follows:
\begin{equation*}
(a_{\rho_i\rho_j})_{i,j=0,1,3,5}=
\begin{pmatrix}
0&1&0&1\\
2&0&1&0\\
0&1&0&1\\
2&0&1&0
\end{pmatrix},
\end{equation*}which determines the special McKay quiver $Q_X$ in an obvious way.

\begin{figure}[h]
\setlength{\unitlength}{1.5pt}
\begin{centering}
\hbox{\hskip 2cm 
\begin{picture}(20,20)(0,0)
\put(0,0){0}
\put(0,10){5}
\put(10,0){1}
\put(-3,-3){
\put(0,0){\line(1,0){10}}
\put(10,0){\line(0,1){10}}
\put(10,10){\line(-1,0){10}}
\put(0,10){\line(0,-1){10}}
\put(-5,-15){$k=0$}
}
\end{picture}
\quad 
\begin{picture}(30,80)(0,0)
\put(0,0){1}
\put(10,0){2}
\put(20,0){3}
\put(0,10){6}
\put(10,10){7}
\put(-1,20){11}
\put(10,20){0}
\put(0,30){4}
\put(0,40){9}
\put(0,50){2}
\put(0,60){7}
\put(0,70){0}
\put(-3,-3){
\put(0,0){\line(0,1){70}}
\put(0,0){\line(1,0){20}}
\put(0,20){\line(0,1){20}}
\put(0,70){\line(1,0){10}}
\put(10,20){\line(0,1){50}}
\put(10,20){\line(1,0){10}}
\put(20,0){\line(0,1){20}}
}
\put(-2,-18){$k=1$}
\end{picture}
\quad
\begin{picture}(30, 30)(0, 0)
\put(0,0){3}
\put(10,0){4}
\put(20,0){5}
\put(0,10){8}
\put(10,10){9}
\put(0,20){1}
\put(-3,-3){
\put(0,0){\line(0,1){20}}
\put(0,20){\line(1,0){20}}
\put(20,20){\line(0,-1){20}}
\put(20,0){\line(-1,0){20}}
\put(0,-15){$k=3$}
}
\end{picture}
\quad
\begin{picture}(80, 30)(0,0)
\put(0,0){5}
\put(10,0){6}
\put(20,0){7}
\put(30,0){8}
\put(40,0){9}
\put(49,0){10}
\put(59,0){11}
\put(70,0){0}
\put(-1,10){10}
\put(9,10){11}
\put(20,10){0}
\put(0,20){3}
\put(-3,-3){
\put(0,0){\line(1,0){70}}
\put(70,0){\line(0,1){10}}
\put(70,10){\line(-1,0){50}}
\put(20,10){\line(0,1){10}}
\put(20,20){\line(-1,0){20}}
\put(0,20){\line(0,-1){20}}
\put(25,-15){$k=5$}
}
\end{picture}
}
\vskip 1cm 
\caption{$O_U\otimesk \rho_k/J(\rho_k)$}
\label{fig:J0135}
\end{centering}
\end{figure}
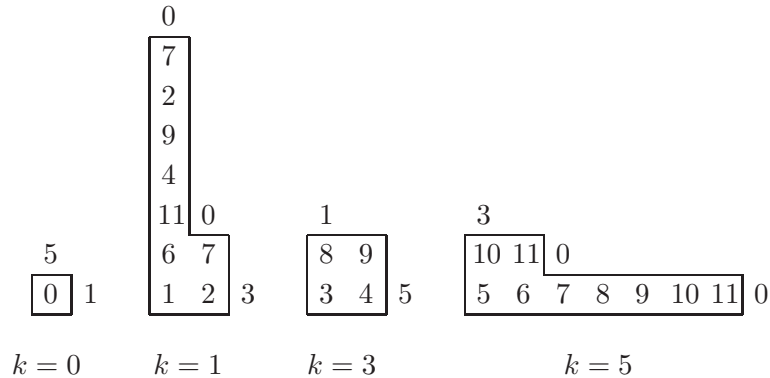

Now we explain Theorem~\ref{thm:cluster-quiver} in the case of this example.
Suppose $E(\rho_3) \cap E(\rho_5)=\{y\}$.
The structure sheaf $O_{Z_y}$ of the $G$-cluster $Z_y$ 
has a $k$-basis consisting of 12 monomials framed in  Figure \ref{fig:OZyy'}.
We note $O_{Z_y}\simeq k[G]\simeq\oplus_{i=0}^{11}\rho_i$.
Let $V(\rho_i)$ be the $\rho_i$-part of $O_{Z_y}$ for $i=3, 5$.
By a direct computation, we see that the ideal $O_UV(\rho_i)$ of $O_{Z_y}$ 
contains no other special representation for $i=3,5$ as in Theorem \ref{thm:socle-special}.
Moreover, $W(\rho_3)$, $W(\rho_0)$ and $W(\rho_5)$ are located 
at the corners of the complement
of the framed part in Figure \ref{fig:OZyy'}.
Then it is easy to observe that
$(\wV(\rho_3) O_U+\fm I_y)/\fm I_y$ contains $W(\rho_5)$
and $(\wV(\rho_5)O_U + \fm I_y)/\fm I_y$ contains $W(\rho_3)$
as in Theorem \ref{thm:cluster-quiver} (1). 
\par
Since $y\in E(\rho_5)$, we have 
$(\wV(\rho_5) O_U+\fm I_y)/\fm I_y$ 
contains $W(\rho_0)$, which gives rise to   
one path of the special McKay quiver from $\rho_5$ to $\rho_0$, that is, an $O_U$-$G$-homomorphism from 
from $O_U\otimesk\rho_5$ to $O_U\otimesk\rho_0$. 
The point $y'$ corresponding to $Z_{y'}$ in Figure \ref{fig:OZyy'} 
is another point of $E(\rho_5)$, 
which gives rise to another path from $\rho_5$ to $\rho_0$.
These two paths are precisely 
the two paths in Figure \ref{fig:J0135} corresponding to
$\phi_{5,0}^1$ and $\phi_{5,0}^2$.
This explains Theorem \ref{thm:cluster-quiver} (2) and 
Proposition~\ref{prop:basis}.\par
Finally, Figure \ref{fig:J0135} shows that 
there is an arrow from $\rho_0$ to $\rho_1$ (resp.  from $\rho_0$ to $\rho_5$).
No arrow from $\rho_0$ to $\rho_5$ 
appears at $y' \in E(\rho_5)$, while an arrow from $\rho_0$ to $\rho_5$ 
does appear at $y \in E(\rho_5)$,
as in the third case in Theorem \ref{thm:cluster-quiver} (3a).
The case $\sigma=\rho_1$ is similar.
If $\sigma \ne \rho_1, \rho_5$,  then 
the first case of Theorem \ref{thm:cluster-quiver} (3a) occurs.

\begin{figure}
\setlength{\unitlength}{1.5pt}
\begin{centering}
\hbox{\hskip 1.5cm
\begin{picture}(60,50)(0,0)
\put(0,0){0}
\put(10,0){1}
\put(20,0){2}
\put(30,0){3}
\put(40,0){4}
\put(50,0){5}
\put(0,10){5}
\put(10,10){6}
\put(20,10){7}
\put(30,10){8}
\put(40,10){9}
\put(-1,20){10}
\put(9,20){11}
\put(20,20){0}
\put(0,30){3}
\put(-3,-3){
\put(0,1){\line(1,0){50}}
\put(50,1){\line(0,1){20}}
\put(50,21){\line(-1,0){28}}
\put(22,21){\line(0,1){10}}
\put(22,31){\line(-1,0){22}}
\put(0,31){\line(0,-1){30}}
\put(22,-15){$O_{Z_y}$}}
\end{picture}
\qquad
\begin{picture}(60,60)(0,0)
\put(0,0){0}
\put(10,0){1}
\put(20,0){2}
\put(30,0){3}
\put(40,0){4}
\put(50,0){5}
\put(60,0){6}
\put(70,0){7}
\put(80,0){8}
\put(90,0){9}
\put(99,0){10}
\put(109,0){11}
\put(120,0){0}
\put(0,10){5}
\put(-3,-3){
\put(0,1){\line(1,0){122}}
\put(122,1){\line(0,1){10}}
\put(122,11){\line(-1,0){121}}
\put(0,11){\line(0,-1){10}}
\put(50,-15){$O_{Z_{y'}}$}}
\end{picture}
}
\vskip 1cm 
\caption{$O_{Z_{y}}$ and $O_{Z_{y'}}$}
\label{fig:OZyy'}
\end{centering}
\end{figure}
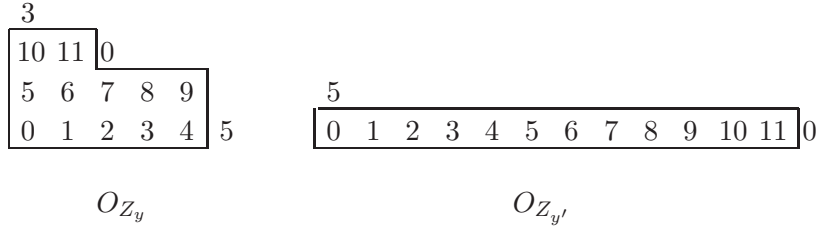

\subsection{The binary dihedral case $D_5$}
\label{subsec:dihedral case}
Let $S=k[x, y]$ and $U=\bA^2_k=\Spec S$. 
The simple singularity $D_5$ is the quotient 
singularity of $U$ by the
binary dihedral group $G:=\bD_3$ of order 12, which 
is generated by $\sigma$ and $\tau$:
 \[
 \sigma=\begin{pmatrix} \epsilon& 0\\ 0& \epsilon^{-1} \end{pmatrix},\quad
 \tau=\begin{pmatrix} 0& 1\\ -1& 0 \end{pmatrix},
 \]
where $\epsilon:=e^{2\pi i/6}$. We have $\sigma^6=\tau^4=1$,
$\sigma^3=\tau^2$ and $\tau\sigma\tau^{-1}=\sigma^{-1}$.  
The group $G$ acts on $U$  from the right by 
$(x,y)\mapsto (x,y)g$ for $g\in G$, hence,   
$\tau$ acts on the ring $S$ is 
defined to be 
$\tau(x)=-y$ and $\tau(y)=x$.  The ring of
$G$-invariants in $S$ is generated by three elements
$$A_6:=x^6+y^6,\ A_8:=xy(x^6-y^6),\ A_4:=x^2y^2.$$ The quotient
$U/G$ is isomorphic to the hypersurface $4A_4^4+A_8^2-A_4A_6^2=0$.

We quote some of the data from \cite[p.~200]{ItoNakamura99}.
By \cite[Table~8, p.~200]{ItoNakamura99}, 
The list of possible generators of $\Gen(I_y)$ $(y\in E)$ is given by    
\begin{align*}
V_4(\rho_0)\oplus V_6(\rho_0)&=\{x^2y^2\}\oplus\{x^6+y^6\},\\
V_2(\rho_1)\oplus V_6(\rho_1)&=\{xy\}\oplus\{x^6-y^6\},\\
V_3(\rho_2)\oplus V_5(\rho_2)&=\{x^2y, -xy^2\}\oplus\{y^5, -x^5\},\\
V_4(\rho_3)&=\{y^4, x^4\}\oplus\{x^3y, -xy^3\},\\
V_3(\rho_4)\oplus V_5(\rho_4)&=\{x^3+ iy^3\}
\oplus\{xy(x^3- iy^3)\},\\
V_3(\rho_5)\oplus V_5(\rho_5)&=\{x^3- iy^3\}
\oplus\{xy(x^3+ iy^3)\}.
\end{align*}
 
Let $S_1$ be the space spanned by $x$ and $y$.
Then multiplication by any element of $S_1$ 
defines an endomorphism of $\Coinv(D_5)$, 
which induces a homomorphism 
from an irreducible factor of $\Coinv(D_5)$
to another irreducible factor. 
Let $y\in F\subset \GHilb(U)$. Let 
$\rho\simeq V(\rho)\subset \Soc(O_{Z_y})$ and 
$\rho'\simeq W(\rho')\subset\Gen(I_y)$. 
If $W(\rho')\subset S_1V(\rho)$, 
we draw a directed arrow from $V(\rho)$ to $W(\rho')$. 
As was explained in \cite[p.~201]{ItoNakamura99}, if we consider  
all pairs $\rho,\sigma\in\Irr(G)$ nontrivial, 
 this gives the Dynkin diagram ${D}_5$ 
with pairs of directed arrows 
in opposite direction as part of   
Figure~\ref{fig:extended quiver D5}. 
See \cite[p.~201]{ItoNakamura99} and \cite{Nakamura09} for details. \par

\begin{figure}[h] 
\setlength{\unitlength}{1.2pt}%
\begin{centering}
\begin{picture}(120,60)(15,25)
\put(-7,60){$V_4(\rho_0)$}
\put(-0,50){$\oplus$}
\put(-7,40){$V_6(\rho_0)$}
\put(-7,20){$V_2(\rho_1)$}
\put(-0,10){$\oplus$}
\put(-7,0){$V_6(\rho_1)$}
\put(42,40){$V_3(\rho_2)$}
\put(53,30){$\oplus$}
\put(42,20){$V_5(\rho_2)$}
\put(92,30){$V_4(\rho_3)$}
\put(141,60){$V_3(\rho_4)$}
\put(148,50){$\oplus$}
\put(141,40){$V_5(\rho_4)$}
\put(141,20){$V_3(\rho_5)$}
\put(148,10){$\oplus$}
\put(141,0){$V_5(\rho_5)$}
 \thicklines
 \put(39,42){\vector(-3,1){15}}
 \put(24,23){\vector(3,1){14}}
 \put(39,23){\vector(-3,-1){15}}
 \put(70,34){\vector(1,0){17}}
 \put(87,30){\vector(-1,0){17}}
 \put(133,45){\vector(-3,-1){15}}
 \put(119,36){\vector(3,1){14}}
 \put(133,23){\vector(-3,1){14}}
 \put(118,23){\vector(3,-1){15}}
\end{picture}
\end{centering}
\vspace{40pt}
\caption{Directed arrows of ${\wD}_5$}
\label{fig:extended quiver D5}
\end{figure}
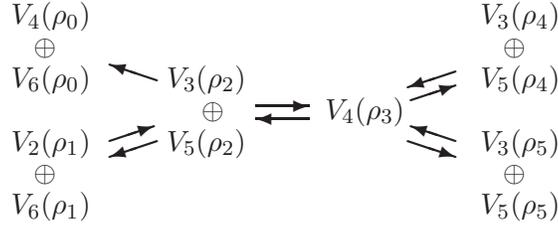

In what follows, we explain what happens at the vertex 
$\rho_0$ of the extended Dynkin diagram $\wD_5$. 
The result looks somewhat novel.  
It is slightly different 
from \cite[p. 277 fifth line from below; p.~278, Fig.~6
\footnote{The arrow from $V_4(\rho_0)$ to $V_3(\rho_2)$ 
in \cite[Fig.~6]{Nakamura09} has to be reversed. 
See \cite[Fig.~14]{Nakamura09}.}]{Nakamura09}. 
\par
Let $V^{\sharp}(\rho_0)=V_4(\rho_0)\oplus V_6(\rho_0)$, and  
$V^{\sharp}(\rho_2)=V_4(\rho_0)\oplus V_6(\rho_0)$. Recall 
$E(\rho_2)=\bP(V^{\sharp}(\rho_2))$. Let 
$W\ (\simeq\rho_2)\subset V^{\sharp}(\rho_2)$, 
and $x=I_x:=I(W)$ any point of $E(\rho_2)$. Then 
\begin{align*}
\Soc(O_{Z_x})[\rho_2]&\simeq V_3(\rho_2)\oplus V_5(\rho_2)/W,\\
\left(S_1\cdot\left(\Soc(O_{Z_x})[\rho_2]\right)\right)[\rho_0]&
\simeq \left(S_1V_3(\rho_2)\oplus 
S_1V_5(\rho_2)\right)[\rho_0]/(S_1W)[\rho_0],\\
&\simeq V^{\sharp}/V^{\sharp}\cap (S_1W)\simeq\Gen(I_x)[\rho_0],
\end{align*}where $\Soc(O_{Z_x})=[I(W):\fm]/\fm I(W)$, 
and $S_1\Soc(O_{Z_x})=S_1([I_x:\fm]/I_x)
\subset I_x/\fm I_x=\Gen(I_x)$. 
This defines an isomorphism from $E(\rho_2)$ 
to 
$\bP(V^{\sharp}(\rho_0))$ sending $W$ to  
$V^{\sharp}/V^{\sharp}\cap (S_1W)$.
Note that there is one more arrow from $\rho_0$ to $\rho_2$
in the McKay quiver, which is understood as an arrow from
$V_0(\rho_0)$ to $V_1(\rho_2)$ as in Theorem \ref{thm:cluster-quiver} (3b).

\begin{subrem}For every other rational double singularity
A similar structure of the coinvariant algebra 
is observed by using Subsection~\ref{subsec:socle OU/mIy} and 
\cite[Tables~7, 10, 13, 17]{ItoNakamura99}. 
See also \cite[p.~289, Fig.~14]{Nakamura09}. 
In the $D_4$ case, the subspace 
$V'_4(\rho_0)=\{x^2y^2\}$ (resp. $V''_4(\rho_0)=\{x^4+y^4\}$ 
plays the same role as $V_4(\rho_0)$ (resp. $V_6(\rho_0)$) of $D_5$. 
In the $A_n$ case $(n\geq 2)$, there is an arrow 
$\rho_1$ to $\rho_0$ (resp.  from $\rho_n$ to $\rho_0$) 
as above, while the other pairs 
are opposite. 
\end{subrem}

\begin{subrem}In the $A_1$ case,
 there is only one arrow from $\rho_1$ to $\rho_0$. 
Indeed, let $S_k$ be the space of polynomials of degree $k$ $(k=1,2)$. 
Then $E(\rho_1)=\bP(S_1)$. Let $x=I(W)$ be any point of $E(\rho_1)$. 
Then $I_x=O_UW$, $\Soc(O_{Z_x})=S_1/W$. Hence 
we have a morphism $f$ from $\bP(S_1)$ to $\bP(S_2)$ sending $W$ to 
$S_2/S_1W=S_1(S_1/W)$. This induces an isomorphism from $\bP(S_1)$ to a nonsingular conic of $\bP(S_2)$. 
\end{subrem}



\end{document}